\documentclass[letterpaper,12pt]{article}
\usepackage{amsmath,amssymb,amsfonts,epsf,epsfig}
\newtheorem{theorem}{Theorem}
\newtheorem{lemma}[theorem]{Lemma}
\newtheorem{notation}[theorem]{Notation}
\newtheorem{prop}[theorem]{Proposition}
\newtheorem{prob}[theorem]{Problem}
\newtheorem{definition}[theorem]{Definition}

\usepackage[shortlabels]{enumitem}
\usepackage{subcaption}
\usepackage{ragged2e}
\usepackage{dsfont}
\usepackage{float}
\usepackage{multicol}
\usepackage{tikz}
\usepackage{array}
\usepackage{enumitem}
\tikzstyle{vertex}=[circle, draw=black, fill=black, minimum size=2pt, inner sep=1.5]
\usetikzlibrary{arrows}
\usetikzlibrary{shapes.geometric}

\newenvironment{proof}{\noindent {\sc Proof}.}
                {\phantom{a} \hfill \framebox[2.2mm]{ } \bigskip}

\usepackage{algorithm,algpseudocode,float}
\usetikzlibrary{arrows.meta,calc,decorations.markings,math,arrows.meta}
\renewcommand{\algorithmicrequire}{\textbf{Input:}}
\renewcommand{\algorithmicensure}{\textbf{Output:}}
\algnewcommand{\algorithmicand}{\textbf{ and }}
\algnewcommand{\algorithmicor}{\textbf{ or }}
\algnewcommand{\algAnd}{\algorithmicand}
\algnewcommand{\algOr}{\algorithmicor}
\newcommand{\minus}{\scalebox{0.75}[1.0]{$-$}}
\newcommand{\midarrow}{\tikz \draw[-{Stealth[scale=1.1]}] (0,0) -- (0.1,0);}
\newcommand{\diff}{\textrm{diff}}
\newcommand{\len}{\textrm{len}}
\definecolor{color1}{rgb}{1,0, 0.5} 
\definecolor{color2}{rgb}{0, 0, 0}
\definecolor{color3}{rgb}{1, 0, 0}
\definecolor{color4}{rgb}{0, 0, 1}
\definecolor{color5}{rgb}{0, 1, 1}
\definecolor{color6}{rgb}{0.98, 0.63, 0.89}
\definecolor{color7}{rgb}{0.01, 0.75, 0.24}

\newcommand{\NN}{\mathbb{N}}
\newcommand{\ZZ}{\mathbb{Z}}
\newcommand{\QQ}{\mathbb{Q}}
\newcommand{\RR}{\mathbb{R}}
\def\sign{{\rm sign}}
\def\supp{{\rm supp}}
\def\length{{\rm length}}
\def\indeg{{\rm indeg}}
\def\outdeg{{\rm outdeg}}
\def\next{{\rm next}}
\def\m{{\mu}}
\def\ch{{\rm children}}
\newcommand{\D}{{\mathcal{D}}}
\newcommand{\E}{{\mathcal{E}}}
\newcommand{\W}{{\mathcal{W}}}
\newcommand{\B}{{\mathcal{B}}}
\newcommand{\A}{{\mathcal{A}}}
\renewcommand{\SS}{{\mathcal{S}}}
\renewcommand{\O}{{\mathcal{O}}}
\renewcommand{\S}{{\Omega}}
\renewcommand{\P}{{\mathcal{P}}}
\def\lvec{\overleftarrow}
\def\rvec{\overrightarrow}

\title{Completing the solution of the directed Oberwolfach  \\ problem with cycles of equal length}
\author{Alice Lacaze-Masmonteil\footnote{Email: alaca054@uottawa.ca. Mailing address: Department of Mathematics and Statistics, University of Ottawa,150 Louis-Pasteur Private, Ottawa, ON, K1N 9A7, Canada.}, University of Ottawa}

\begin{document}
\maketitle \baselineskip 17pt

\begin{center}
{\bf Abstract}
\end{center}

In this paper, we give a solution to the last outstanding case of the directed Oberwolfach problem with tables of uniform length. Namely, we address the two-table case with tables of equal odd length.  We prove that the complete symmetric digraph on $2m$ vertices, denoted $K^*_{2m}$, admits a resolvable decomposition into directed cycles of odd  length $m$. This completely settles the directed Oberwolfach problem with tables of uniform length. \\


\section{Introduction}

In this paper, we address the last open case of the directed Oberwolfach problem with tables of uniform length, namely the case with two tables of odd length. A variation of the celebrated Oberwolfach problem, the directed Oberwolfach problem asks whether $t$ conference attendees can be seated at $k$ round tables seating $m_1, m_2, \ldots, m_k$ guests, respectively, over the course of $t-1$ nights, with the crux being that each guest is to be seated to the right of every other guest exactly once. In Problem \ref{prob:ini} below, we formulate this problem in graph-theoretic terms for the case of equal-size tables.

\begin{prob}
\label{prob:ini}
Let $\alpha$ and $m$ be positive integers. Identify all values of $\alpha$ and $m$ for which $K^*_{\alpha m}$ admits a resolvable decomposition into directed cycles of length $m$. 
\end{prob}

Observe that a resolvable decomposition of $K^*_{\alpha m}$ into directed cycles of length $m$ is equivalent to a resolvable Mendelsohn design with blocks of size $m$ \cite{ColDin}. 

We first point out that a solution to the original Oberwolfach problem with tables of uniform length can be found in \cite{AlsHag, AlsSch, HofSch, HuaKot}.  In \cite{AlsGavSaj}, Alspach et al. show that $K^*_n$ admits a decomposition into directed cycles of length $m$  if and only if $m$ divides the number of arcs in $K^*_n$ and $(n,m) \not \in \{(4,4), (6,3), (6,6)\}$.  

Problem \ref{prob:ini} has been solved in the following cases. In \cite{BerGerSot}, Bermond et al. proved that $K^*_{3\alpha}$ admits a resolvable decomposition into directed cycles of length 3 if and only if $\alpha \neq 2$. Then, Bennett and Zhang \cite{BenZha} and Adams and Bryant \cite{AdaBry} jointly showed that $K^*_{4\alpha}$ admits a resolvable decomposition into directed cycles of length 4 if and only if $\alpha \neq 1$. Abel et al. \cite{Abel} settled the existence of a resolvable decomposition of $K^*_{5\alpha}$ into directed cycles of length 5 for $\alpha \geqslant 103$, with a few possible exceptions.  In \cite{Til}, Tillson settled the existence of a directed Hamiltonian decomposition of $K^*_m$ for $m$ even. An important step was taken by Burgess and \v{S}ajna \cite{BurSaj}, who solved Problem \ref{prob:ini} for the cases $m$ is even, and when $\alpha$ and $m$ are odd. Regarding the case $\alpha$ even and $m$ odd, Burgess and \v{S}ajna showed that if $K^*_{2m}$ admits a resolvable decomposition into directed cycles of length $m$, then $K^*_{\alpha m}$ also admits a resolvable decomposition into directed cycles of length $m$. Therefore, to completely resolve the directed Oberwolfach problem with cycles of uniform length, it suffices to construct a resolvable decomposition of $K_{2m}^*$ into directed cycles of length $m$ for $m$ odd, $m\geqslant 5$. This problem has proven to be particularly difficult and has thus far only been solved for a finite number of cases. 

\begin{theorem}  \cite{BurFranSaj} \label{BurFraSaj}
Let  $m$ be an odd integer such that $5 \leqslant m \leqslant 49$. The digraph $K^*_{2m}$ admits a resolvable decomposition into directed cycles of length $m$. 
\end{theorem}

The main result of this paper is stated in Theorem \ref{thm:main} below.

\begin{theorem}
\label{thm:main}
Let $m$ be an odd integer such that $m \geqslant 5$. The digraph $K^*_{2m}$ admits a resolvable decomposition into directed cycles of length $m$.
\end{theorem}

Theorem \ref{thm:main}, in conjunction with results from  \cite{Abel, AdaBry, BenZha, BerGerSot, BurFranSaj, BurSaj, Til}, implies a complete resolution of the directed Oberwolfach problem with tables of uniform length stated as Theorem \ref{mains} below. 

\begin{theorem}
\label{mains}
The digraph $K^*_{\alpha m}$ admits a resolvable decomposition into directed cycles of length $m$ if and only if $(\alpha, m) \not\in \{ (1,6), (1, 4),$ $(2, 3) \}$. 
\end{theorem}

We now provide an outline of this paper. In Section 2, we give key definitions. In Sections 3 and 4, we show that three particular classes of digraphs on $2m$ vertices admit a resolvable decomposition into directed cycles of odd length $m$. These decompositions are then used in Section 5 to construct a resolvable decomposition of $K^*_{2m}$ into directed $m$-cycles. 

\section{Preliminaries}

In this paper, all directed graphs (digraphs for short) are strict, meaning that they do not contain loops or parallel arcs. If $G$ is a digraph (graph), we shall denote its vertex set as $V(G)$ and its arc set (edge set) as $A(G)$ ($E(G)$, respectively). We denote the complete graph on $m$ vertices as $K_m$. The \textit{complete symmetric digraph of order $m$}, denoted $K^*_m$, is the strict digraph on $m$ vertices such that for any two distinct vertices $x$ and $y$, we have $(x, y), (y, x) \in A(K^*_m)$.  The symbol $\vec{C}_m$ denotes the directed cycle on $m$ vertices. We shall denote the length, that is, the number of arcs, of a directed path (dipath for short) $P$ as $\len(P)$. The first vertex of a dipath $P$ is known as its \textit{source} and is denoted $s(P)$, while the last vertex of $P$ is known as its \textit{terminal} and is denoted $t(P)$. The \textit{concatenation} of dipaths $P=v_1v_2\ldots v_n$ and $Q=v_nv_{n+1}\ldots v_m$ is the directed walk $PQ=v_1v_2\ldots v_{n-1}v_nv_{n+1}\ldots v_m$. The digraph that consists of $k$ disjoint copies of a digraph $G$ is denoted $kG$. 

\begin{definition} \rm
\label{def:1}
Let $G$ be a digraph. A set $\{H_1, H_2, \ldots, H_k\}$ of subdigraphs of $G$ is a \textit{decomposition of $G$} if $\{A(H_1), A(H_2), \ldots, A(H_k)\}$ is a partition of $A(G)$. If $G$ admits such a decomposition, we write $G=H_1\oplus H_2 \oplus \ldots \oplus H_k$. 
\end{definition}

\begin{definition}\rm
Let $G$ be a digraph. A \textit{directed 2-factor} of $G$ is a spanning subdigraph of $G$ comprised of disjoint directed cycles of $G$. A \textit{$\vec{C}_m$-factor} is a directed 2-factor of $G$ in which all directed cycles are of length $m$. A \textit{$\vec{C}_m$-factorization} is a decomposition of $G$  into $\vec{C}_m$-factors. 
\end{definition}

If a digraph $G$ admits a $\vec{C}_m$-factorization, then $|V(G)|$ and $|A(G)|$ are necessarily multiples of $m$. Hence, Problem \ref{prob:ini} asks to identify all values of $m$ for which this condition for $K^*_{\alpha m}$ is also sufficient.  

\begin{definition} \rm
The \textit{wreath product} of digraphs $G$ and $H$, denoted $G\wr H$, is the digraph with vertex set $V(G) \times V(H)$ and an arc from $(g_1, h_1)$ to $(g_2, h_2)$ if and only if $(g_1, g_2) \in A(G)$, or $g_1=g_2$ and $(h_1, h_2) \in A(H)$. 
\end{definition} 

\begin{definition} \rm
Let $m$ be a positive integer and $C \subseteq \{1,2, \ldots, \lfloor \frac{m}{2} \rfloor\}$. The \textit{circulant of order $m$}  with \textit{connection set} $C$,  denoted $X(m, C)$, is the graph with vertex set $V(G)=\mathds{Z}_m$ and edge set $E(G)=\{\{x, y\} \, |\, \min\{x-y, m-(x-y)\} \in C\}$, with $x-y$ evaluated modulo $m$. 

Let $D \subseteq \{1,2, \ldots, m-1\}$. The \textit{directed circulant of order $m$} with \textit{connection set} $D$, denoted $\vec{X}(m, D)$, has vertex set $V(G)=\mathds{Z}_m$ and arc set $A(G)=\{(x, y)\, | \, y-x \in D\}$, with $y-x$ evaluated modulo $m$.
\end{definition}

\section{Decomposition of $\vec{X}(m, \{\pm1\}) \wr \overline{K}_2$ and $\vec{X}(m, \{1, 3\}) \wr \overline{K}_2$}

In the following two sections, we will show that three special classes of digraphs admit a $\vec{C}_m$-factorization. First, we introduce some definitions and notation pertaining to these digraphs. 

\begin{notation}
\label{not:2di}
Let $m$ be an odd integer. Let

\medskip
{\centering
  $ \displaystyle
    \begin{aligned} 
 &H_{2m}=\vec{X}(m, \{\pm1\})\, \wr  \,\overline{K}_2,\\ 
 &L_{2m}=\vec{X}(m, \{1, 3\})\, \wr \, \overline{K}_2, \textrm{and} \\
 &G_{2m} =\vec{X}(m, \{1, 3\})\, \wr \, K^*_2.  
    \end{aligned}
    $
\par
\medskip}

\noindent We shall assume that 
\begin{center}
$V(\overline{K}_2)=V(K^*_2)=\{x, y\}$ and $V(H_{2m})=V(L_{2m})=V(G_{2m})=\{ x_a, y_b\, |\ a,b \in \mathds{Z}_m\}$ 
 \end{center}
 where $ x_a=(a,x)$ and $y_b=(b, y)$.
\end{notation}

In Figure \ref{fig:defini}, we illustrate the digraph  $L_{22}$. Note that, in all figures, the edges are assumed to be arcs oriented from left to right. The orientation of the vertical arcs will be clear from the context. 

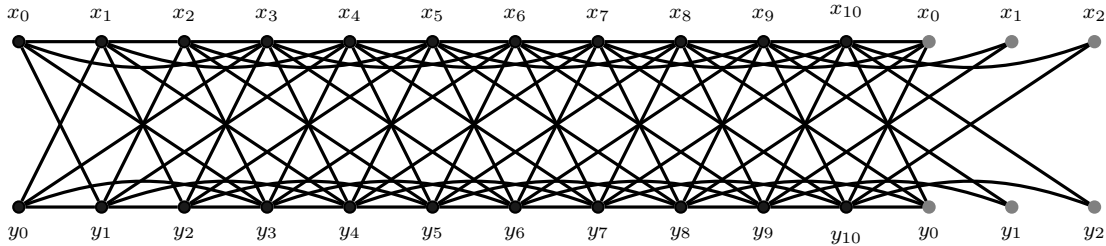
\begin{figure} [htpb] 

\begin{tikzpicture}[thick,  every node/.style={circle,draw=black,fill=black!90, inner sep=1.5}, scale=1.1]
 \node (x0) at (0.0,2.0) [label=above:$\scriptstyle x_0 $] {};
 \node (x1) at (1.0,2.0) [label=above:$\scriptstyle x_1 $] {};
 \node (x2) at (2.0,2.0) [label=above:$\scriptstyle x_2$] {};
 \node (x3) at (3.0,2.0) [label=above:$\scriptstyle x_3 $] {};
 \node (x4) at (4.0,2.0) [label=above:$\scriptstyle x_4 $] {};
 \node (x5) at (5.0,2.0) [label=above:$\scriptstyle x_5 $] {};
 \node (x6) at (6.0,2.0) [label=above:$\scriptstyle x_6 $] {};
 \node (x7) at (7.0,2.0) [label=above:$\scriptstyle x_7 $] {};
 \node (x8) at (8.0,2.0) [label=above:$\scriptstyle x_8 $] {};
 \node (x9) at (9.0,2.0) [label=above:$\scriptstyle x_9 $] {};
 \node (x10) at (10.0,2.0) [label=above:$\scriptstyle x_{10} $] {};
 \node (x11) at (11,2.0)[draw=gray, fill=gray] [label=above:$\scriptstyle x_{0} $] {};
 \node (x12) at (12,2.0)[draw=gray, fill=gray][label=above:$\scriptstyle x_{1} $] {};
 \node (x13) at (13,2.0) [draw=gray, fill=gray][label=above:$\scriptstyle x_{2} $] {};
\node(y0) at (0,0)  [label=below:$\scriptstyle y_0$]{};
\node(y1) at (1,0)  [label=below:$\scriptstyle y_1$]{};
\node(y2) at (2,0)  [label=below:$\scriptstyle y_2$]{};
\node(y3) at (3,0)  [label=below:$\scriptstyle y_3$]{};
\node(y4) at (4,0)  [label=below:$\scriptstyle y_4$]{};
\node(y5) at (5,0)  [label=below:$\scriptstyle y_5$]{};
\node(y6) at (6,0)  [label=below:$\scriptstyle y_6$]{};
\node(y7) at (7,0)  [label=below:$\scriptstyle y_7$]{};
\node(y8) at (8,0)  [label=below:$\scriptstyle y_8$]{};
\node(y9) at (9,0)  [label=below:$\scriptstyle y_9$]{};
\node(y10) at (10,0)  [label=below:$\scriptstyle y_{10}$]{};
\node(y11) at (11,0)  [draw=gray, fill=gray][label=below:$\scriptstyle y_{0}$]{};
\node(y12) at (12,0)  [draw=gray, fill=gray][label=below:$\scriptstyle y_{1}$]{};
\node(y13) at (13,0) [draw=gray, fill=gray] [label=below:$\scriptstyle y_{2}$]{};

 \path[every node/.style={font=\sffamily\small}]
  (x0) edge [color2, very thick]   (x1)
  (x1) edge [color2, very thick]   (x2)
  (x2) edge [color2, very thick]   (x3)
  (x3) edge [color2, very thick]  (x4)
  (x4) edge [color2, very thick]  (x5)
  (x5) edge [color2, very thick]   (x6)
  (x6) edge [color2, very thick]   (x7)
  (x7) edge [color2, very thick] (x8)
  (x8) edge [color2, very thick]   (x9)
  (x9) edge [color2, very thick]   (x10)
  (x10) edge [color2, very thick]   (x11)
  (y0) edge [color2, very thick]   (y1)
  (y1) edge [color2, very thick]    (y2)
  (y2) edge [color2, very thick]  (y3)
  (y3) edge [color2, very thick]   (y4)
  (y4) edge [color2, very thick]   (y5)
  (y5) edge [color2, very thick]   (y6)
  (y6) edge [color2, very thick]    (y7)
  (y7) edge [color2, very thick]   (y8)
  (y8) edge [color2, very thick]  (y9)
  (y9) edge [color2, very thick]   (y10)
  (y10) edge [color2, very thick]  (y11)

  (x0) edge [color2, very thick]    (y1)
  (x1) edge [color2, very thick]    (y2)
  (x2) edge [color2, very thick]   (y3)
  (x3) edge [color2, very thick]   (y4)
  (x4) edge [color2, very thick]   (y5)
  (x5) edge [color2, very thick]   (y6)
  (x6) edge [color2, very thick]   (y7)
  (x7) edge [color2, very thick]   (y8)
  (x8) edge [color2, very thick]   (y9)
  (x9) edge [color2, very thick]   (y10)
  (x10) edge [color2, very thick]   (y11)
  
  (y0) edge [color2, very thick]   (x1)
  (y1) edge [color2, very thick]    (x2)
  (y2) edge [color2, very thick]    (x3)
  (y3) edge [color2, very thick]    (x4)
  (y4) edge [color2, very thick]    (x5)
  (y5) edge [color2, very thick]    (x6)
  (y6) edge [color2, very thick]    (x7)
  (y7) edge [color2, very thick]    (x8)
  (y8) edge [color2, very thick]    (x9)
  (y9) edge [color2, very thick]    (x10)
  (y10) edge [color2, very thick]    (x11)

  (y0) edge [color2, very thick]   (x3)
  (y1) edge [color2, very thick]   (x4)
  (y2) edge [color2, very thick]   (x5)
  (y3) edge [color2, very thick]   (x6)
  (y4) edge [color2, very thick]   (x7)
  (y5) edge [color2, very thick]   (x8)
  (y6) edge [color2, very thick]   (x9)
  (y7) edge [color2, very thick]   (x10)
  (y8) edge [color2, very thick]   (x11)
  (y9) edge [color2, very thick]   (x12)
  (y10) edge [color2, very thick]   (x13)

  (x0) edge [color2, very thick]   (y3)
   (x1) edge [color2, very thick]   (y4)
    (x2) edge [color2, very thick]   (y5)
     (x3) edge [color2, very thick]   (y6)
      (x4) edge [color2, very thick]   (y7)
       (x5) edge [color2, very thick]   (y8)
        (x6) edge [color2, very thick]   (y9)
         (x7) edge [color2, very thick]   (y10)
          (x8) edge [color2, very thick]   (y11)
           (x9) edge [color2, very thick]   (y12)
            (x10) edge [color2, very thick]   (y13)

  (x0) edge [color2, very thick, bend right=20]   (x3)
   (x1) edge [color2, very thick, bend right=20]   (x4)
  (x2) edge [color2, very thick, bend right=20]   (x5)
  (x3) edge [color2, very thick, bend right=20]   (x6)
  (x4) edge [color2, very thick, bend right=20]   (x7)
  (x5) edge [color2, very thick, bend right=20]   (x8)
  (x6) edge [color2, very thick, bend right=20]   (x9)
  (x7) edge [color2, very thick, bend right=20]   (x10)
  (x8) edge [color2, very thick, bend right=20]   (x11)
  (x9) edge [color2, very thick, bend right=20]   (x12)
  (x10) edge [color2, very thick, bend right=20]   (x13)
   (y0) edge [color2, very thick, bend left=20]   (y3)
   (y1) edge [color2, very thick, bend left=20]   (y4)
  (y2) edge [color2, very thick, bend left=20]   (y5)
  (y3) edge [color2, very thick, bend left=20]   (y6)
  (y4) edge [color2, very thick, bend left=20]   (y7)
  (y5) edge [color2, very thick, bend left=20]   (y8)
  (y6) edge [color2, very thick, bend left=20]   (y9)
  (y7) edge [color2, very thick, bend left=20]   (y10)
  (y8) edge [color2, very thick, bend left=20]   (y11)
  (y9) edge [color2, very thick, bend left=20]   (y12)
  (y10) edge [color2, very thick, bend left=20]   (y13);
  
\end{tikzpicture}
\caption{The digraph $L_{22}$. Arcs are oriented from left to right.}
\label{fig:defini}
\end{figure}
\begin{sloppypar}
When constructing $L_{2m}$ and $G_{2m}$, we chose $\vec{X}(m, \{1,3\})$ as the first factor in the wreath product. The digraph  $\vec{X}(m, \{1,2\})$ would be simpler, however, it can be shown that ${\vec{X}(m, \{1,2\})\, \wr \, \overline{K}_2}$ does not admit a $\vec{C}_m$-factorization. 
\end{sloppypar}

Let $G \in \{H_{2m}, L_{2m}, G_{2m}\}$. Arcs of $G$ of the form $(x_a, x_b)$, $(x_a, y_b)$, $(y_a, y_b)$, and $(y_a, x_b)$ are said to be of \textit{difference $b-a$} (computed modulo $m$). Let $W=v_0 v_1 \ldots v_n$ be a directed walk in $G$, with consecutive arcs of difference $d_0, d_1, \ldots, d_{n-1}$. If  $t=d_0+d_1+\ldots+d_{n-1}$, then we say that the arcs of $W$ \textit{sum up to $t$}. A \textit{type-$k$} cycle of $G$ (for $k$ a positive integer) is a directed $m$-cycle whose arcs sum up to $km$. It follows from the definition of $L_{2m}$ and $G_{2m}$ that the arcs of a directed $m$-cycle of these digraphs sum up to at most $3m$. 

In Lemmas \ref{lem:1} and \ref{lem:2}, and Proposition \ref{thm:cand}, we will construct a $\vec{C}_m$-factorization of $H_{2m}, L_{2m}$, and $G_{2m}$, respectively, containing four, four, and five $\vec{C}_m$-factors, respectively.

\begin{lemma}
\label{lem:1}
Let $m \geqslant 5$ be an odd integer. The digraph $H_{2m}$ admits a $\vec{C}_m$-factorization. 
\end{lemma}
\begin{proof}
We construct eight directed $m$-cycles of $H_{2m}$ as follows:

\begin{multicols}{2}
\medskip
{\centering
  $ \displaystyle
    \begin{aligned} 
&C^0=x_0 x_1x_2x_3x_4x_5\ldots x_{m-3} x_{m-2} x_{m-1} x_0; \\
&C^2=x_0 y_1 y_2 x_3y_4x_5 \ldots y_{m-3} x_{m-2} y_{m-1} x_0; \\
&C^4=y_0 y_1x_2 y_3x_4y_5\ldots x_{m-3} y_{m-2} x_{m-1} y_0; \\
&C^6=y_0 x_1 y_2 y_3y_4y_5 \ldots y_{m-3} y_{m-2} y_{m-1} y_0; \\
    \end{aligned}
  $ 
\par}
\medskip
\medskip
{\centering
  $ \displaystyle
    \begin{aligned} 
&C^1= y_0  y_{m-1} y_{m-2} y_{m-3} \ldots y_5y_4 y_3  y_2 y_1 y_0;\\    
&C^3=y_0  x_{m-1} y_{m-2} x_{m-3} \ldots y_5x_4y_3  x_2 x_1 y_0;\\
&C^5= x_0  y_{m-1} x_{m-2} y_{m-3} \ldots x_5y_4 x_3 y_2 x_1x_0;\\
&C^7=x_0 x_{m-1} x_{m-2} x_{m-3} \ldots x_5x_4x_3 x_2 y_1 x_0.\\
    \end{aligned}
  $ 
\par}
\medskip
\end{multicols}
It can be verified that, for each even $i$, the directed $m$-cycles $C^i$ and $C^{i+1}$ are disjoint. Each arc of $H_{2m}$ occurs precisely once in these eight directed $m$-cycles. Therefore, the set $\{C^0 \cup C^1, C^2\cup C^3, C^4\cup C^5, C^6\cup  C^7\}$ is a $\vec{C}_m$-factorization of $H_{2m}$. \end{proof}

We proceed by constructing a $\vec{C}_m$-factorization of $L_{2m}$ for all odd $m\geqslant 13$ that are not divisible by 3. 

\begin{lemma}
\label{lem:2}
Let $m\geqslant 13$ be an odd integer. The digraph $L_{2m}$ admits a $\vec{C}_m$-factorization if and only if $3 \nmid \,m$.
\end{lemma}

\begin{proof} First assume that $3 |m$. Let $C$ be a directed $m$-cycle of $L_{2m}$ containing an arc of difference 3. Observe that $C$ cannot be a type-1 directed $m$-cycle. Now suppose that $C$ is a type-2 directed $m$-cycle comprised of $k_1$ arcs of difference 1 and $k_2$ arcs of difference 3. Then $k_1+k_2=m$ and $k_1+3k_2=2m$, implying that $k_1 \not \equiv k_2\ (\textrm{mod}\ 2)$ and $k_1\equiv k_2\ (\textrm{mod}\ 2)$, respectively --- a contradiction. Hence, if $C$ contains an arc of difference 3, then all arcs of $C$ are of difference $3$. However, if all arcs of $C$ are of difference 3 then $C$ has a repeated vertex --- also a contradiction. It follows that $L_{2m}$ does not admit a $\vec{C}_m$-factorization.

Conversely, assume that $3 \nmid m$. We construct a $\vec{C}_m$-factorization of $L_{2m}$ with four $\vec{C}_m$-factors, each consisting of one type-1 directed $m$-cycle and one type-3 directed $m$-cycle. Let $m=p+6k$ with $p\in \{1,5\}$ and $k\geqslant 2$. 

\noindent {\bf Case 1:} $p=1$. To construct our first $\vec{C}_m$-factor, we start by building eight dipaths as follows (see Figure \ref{fig2}): 
\begin{multicols}{2}
\medskip
{\centering
  $ \displaystyle
    \begin{aligned} 
&W_0=y_0y_1 y_2 x_3 x_4 x_5y_{6} y_{7} x_{8} x_{9} x_{10}y_{11} x_{12} y_{13};\\
&X_0=x_2 y_5 y_{8} x_{11}x_{14};\\
&Y_0=x_1 y_4 x_{7} y_{10}x_{13}; \\
&Z_0=x_0 y_3 x_{6} y_{9}y_{12}x_{15};
    \end{aligned}
  $ 
\par}
\medskip
\medskip
{\centering
  $ \displaystyle
    \begin{aligned} 
&Q_0=y_{13}y_{14} y_{15} \ldots y_{m-2}y_{m-1}y_0;\\
&R_0=x_{14}x_{17} x_{20} \ldots x_{m-5}x_{m-2}x_1;\\
&S_0=x_{13}x_{16} x_{19} \ldots x_{m-6}x_{m-3}x_0; \\
&T_0=x_{15}x_{18} x_{21} \ldots x_{m-4}x_{m-1}x_2.    
    \end{aligned}
  $ 
\par}
\medskip
\end{multicols}

\begin{figure}[h!]
\begin{FlushLeft}
\begin{subfigure}[l]{1 \textwidth}
\begin{tikzpicture}[thick,  every node/.style={circle,draw=black,fill=black!90, inner sep=1.2}, scale=0.65]
 \node (x0) at (0.0,1.0) [label=above:$\scriptstyle x_0 $] {};
 \node (x1) at (1.0,1.0) [label=above:$\scriptstyle x_1 $] {};
 \node (x2) at (2.0,1.0) [label=above:$\scriptstyle x_2$] {};
 \node (x3) at (3.0,1.0) [label=above:$\scriptstyle x_3 $] {};
 \node (x4) at (4.0,1.0) [label=above:$\scriptstyle x_4 $] {};
 \node (x5) at (5.0,1.0) [label=above:$\scriptstyle x_5 $] {};
 \node (x6) at (6.0,1.0) [label=above:$\scriptstyle x_6 $] {};
 \node (x7) at (7.0,1.0) [label=above:$\scriptstyle x_7 $] {};
 \node (x8) at (8.0,1.0) [label=above:$\scriptstyle x_8 $] {};
 \node (x9) at (9.0,1.0) [label=above:$\scriptstyle x_9 $] {};
 \node (x10) at (10.0,1.0) [label=above:$\scriptstyle x_{10} $] {};
 \node (x11) at (11,1.0) [label=above:$\scriptstyle x_{11} $] {};
 \node (x12) at (12,1.0)[label=above:$\scriptstyle x_{12} $] {};
 \node (x13) at (13,1.0) [label=above:$\scriptstyle x_{13} $] {};
 \node (x14) at (14,1.0) [label=above:$\scriptstyle x_{14} $] {};
 \node (x15) at (15,1.0) [label=above:$\scriptstyle x_{15} $] {};
 \node (x16) at (16,1.0) [label=above:$\scriptstyle x_{16} $] {};
 \node (x17) at (17,1.0) [label=above:$\scriptstyle x_{17} $] {};
 \node (x18) at (18,1.0) [label=above:$\scriptstyle x_{18} $] {};
 \node (x19) at (19,1.0) [label=above:$\scriptstyle x_{19} $] {};
  \node (x20) at (20,1.0) [label=above:$\scriptstyle x_{20} $] {};
 \node (x21) at (21,1.0) [label=above:$\scriptstyle x_{21} $] [label=right:$\ldots$]{};
\node(y0) at (0,0)  [label=below:$\scriptstyle y_0$]{};
\node(y1) at (1,0)  [label=below:$\scriptstyle y_1$]{};
\node(y2) at (2,0)  [label=below:$\scriptstyle y_2$]{};
\node(y3) at (3,0)  [label=below:$\scriptstyle y_3$]{};
\node(y4) at (4,0)  [label=below:$\scriptstyle y_4$]{};
\node(y5) at (5,0)  [label=below:$\scriptstyle y_5$]{};
\node(y6) at (6,0)  [label=below:$\scriptstyle y_6$]{};
\node(y7) at (7,0)  [label=below:$\scriptstyle y_7$]{};
\node(y8) at (8,0)  [label=below:$\scriptstyle y_8$]{};
\node(y9) at (9,0)  [label=below:$\scriptstyle y_9$]{};
\node(y10) at (10,0)  [label=below:$\scriptstyle y_{10}$]{};
\node(y11) at (11,0)  [label=below:$\scriptstyle y_{11}$]{};
\node(y12) at (12,0)  [label=below:$\scriptstyle y_{12}$]{};
\node(y13) at (13,0)  [label=below:$\scriptstyle y_{13}$]{};
\node(y14) at (14,0)  [label=below:$\scriptstyle y_{14}$]{};
\node(y15) at (15,0)  [label=below:$\scriptstyle y_{15}$]{};
\node(y16) at (16,0)  [label=below:$\scriptstyle y_{16}$]{};
 \node (y17) at (17,0.0) [label=below:$\scriptstyle y_{17} $] {};
 \node (y18) at (18,0.0) [label=below:$\scriptstyle y_{18} $] {};
 \node (y19) at (19,0.0) [label=below:$\scriptstyle y_{19} $] {};
 \node (y20) at (20,0.0) [label=below:$\scriptstyle y_{20} $] {};
 \node (y21) at (21,0.0) [label=below:$\scriptstyle y_{21} $] [label=right:$\ldots$]{};

 \path[every node/.style={font=\sffamily\small}]
 
 (x0) edge [very thick, color7](y3)
 (y3) edge [very thick, color7](x6)
 (x6) edge [very thick, color7](y9)
 (y9) edge [very thick, color7, bend left=20](y12)
 (y12) edge [very thick, color7](x15)
 (x15) edge [very thick, color7, bend right=20](x18)
 (x18) edge [very thick, color7,  bend right=20](x21)
 
 (x1) edge [very thick, color5](y4)
 (y4) edge [very thick, color5](x7)
 (x7) edge [very thick, color5](y10)
(y10) edge [very thick, color5](x13)
(x13) edge [very thick, color5,  bend right=20](x16)
(x16) edge [very thick, color5, bend right=20](x19)

(x2) edge [very thick, color4](y5)
(y5) edge [very thick, color4, bend left=20](y8)
(y8) edge [very thick, color4](x11)
(x11) edge [very thick, color4, bend right=20](x14)
(x14) edge [very thick, color4,  bend right=20](x17)
(x17) edge [very thick, color4, bend right=20](x20)

(y0) edge [very thick, color3](y1)
(y1) edge [very thick, color3](y2)
(y2) edge [very thick, color3](x3)
(x3) edge [very thick, color3](x4)
(x4) edge [very thick, color3](x5)
(x5) edge [very thick, color3](y6)
(y6) edge [very thick, color3] (y7)
(y7) edge [very thick, color3] (x8)
(x8) edge [very thick, color3](x9)
(x9) edge [very thick, color3](x10)
(x10) edge [very thick, color3](y11)
(y11) edge [very thick, color3](x12)
(x12) edge [very thick, color3](y13)
(y13) edge [very thick, color3](y14)
(y14) edge [very thick, color3](y15)
(y15) edge [very thick, color3](y16)
(y16) edge [very thick, color3](y17)
(y17) edge [very thick, color3](y18)
(y18) edge [very thick, color3](y19);

\end{tikzpicture}
\caption{The dipaths $W_0Q_0$ (red),  $X_0R_0$ (dark blue), $Y_0S_0$ (light blue), and $Z_0T_0$ (green).}
\end{subfigure}
\begin{subfigure}[l]{1.0\textwidth}
\hfill
\begin{tikzpicture}[thick,  every node/.style={circle,draw=black,fill=black!90, inner sep=1.2}, scale=1]
 \node (x0) at (0.0,1.0) [label={[label distance=-0.2cm]above:$\scriptstyle x_{m-6}$}][label=left:$\ldots$] {};
 \node (x1) at (1.0,1.0) [label={[label distance=-0.2cm]above:$\scriptstyle x_{m-5}$}]{};
 \node (x2) at (2.0,1.0)[label={[label distance=-0.2cm]above:$\scriptstyle x_{m-4}$}]{};
 \node (x3) at (3.0,1.0)[label={[label distance=-0.2cm]above:$\scriptstyle x_{m-3}$}]{};
 \node (x4) at (4.0,1.0) [label={[label distance=-0.2cm]above:$\scriptstyle x_{m-2}$}]{};
 \node (x5) at (5.0,1.0)[label={[label distance=-0.2cm]above:$\scriptstyle x_{m-1}$}]{};
 \node (x6) at (6.0,1.0)[draw=gray, fill=gray]  [label=above:$\scriptstyle x_{0}  $] {};
 \node (x7) at (7.0,1.0) [draw=gray, fill=gray] [label=above:$\scriptstyle  x_{1}  $] {};
 \node (x8) at (8.0,1.0) [draw=gray, fill=gray] [label=above:$\scriptstyle  x_{2} $] {};
\node(y0) at (0,0)  [label={[label distance=-0.2cm]below:$\scriptstyle y_{m-6}$}][label=left:$\ldots$] {};
\node(y1) at (1,0)  [label={[label distance=-0.2cm]below:$\scriptstyle y_{m-5}$}]{};
\node(y2) at (2,0)  [label={[label distance=-0.2cm]below:$\scriptstyle y_{m-4}$}]{};
\node(y3) at (3,0) [label={[label distance=-0.2cm]below:$\scriptstyle y_{m-3}$}]{};
\node(y4) at (4,0)  [label={[label distance=-0.2cm]below:$\scriptstyle y_{m-2}$}]{};
\node(y5) at (5,0)  [label={[label distance=-0.2cm]below:$\scriptstyle y_{m-1}$}]{};
\node(y6) at (6,0)  [draw=gray, fill=gray]   [label=below:$\scriptstyle y_{0}$]{};
\node(y7) at (7,0) [draw=gray, fill=gray]  [label=below:$\scriptstyle y_{1}$]{};
\node(y8) at (8,0) [draw=gray, fill=gray]  [label=below:$\scriptstyle y_{2}$]{};

 \path[every node/.style={font=\sffamily\small}]
 
  (x1) edge [very thick, color4, bend right=20](x4)
 (x4) edge [very thick, color4,  bend right=20](x7)
 
 (x2) edge [very thick, color7,  bend right=20](x5)
 (x5) edge [very thick, color7,  bend right=20](x8)
 
(x0) edge [very thick, color5,  bend right=20](x3)
(x3) edge [very thick, color5,  bend right=20](x6)
(y0) edge [very thick, color3](y1)
(y1) edge [very thick, color3](y2)
(y2) edge [very thick, color3](y3)
(y3) edge [very thick, color3](y4)
(y4) edge [very thick, color3](y5)
(y5) edge [very thick, color3](y6);

\end{tikzpicture}

\caption{The last six arcs of $Q_0$ (red), and the last three arcs of  $R_0$ (dark blue), $S_0$ (light blue), and $T_0$ (green).}
\label{fig:m5.1.2}
\end{subfigure}
\end{FlushLeft}
\caption{The key dipaths in the construction of $F_0$ for $p=1$.}
\label{fig2}
\end{figure}
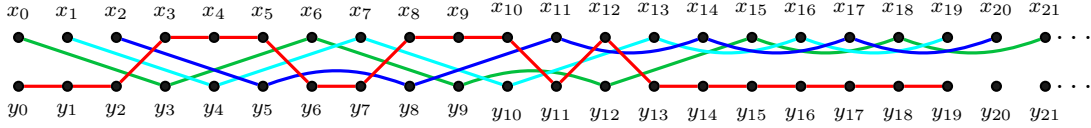
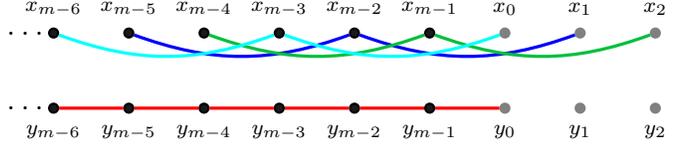
\pagebreak
The dipath $Q_0$ is of length $m-13$ while the dipaths in $\{R_0, S_0, T_0\}$ are each of length $2(k-2)$. We also point out that the dipaths in $\{Q_0, R_0, S_0, T_0\}$ are pairwise disjoint. We can then form the following directed $m$-cycles:

\begin{center}
$C^0=W_0Q_0$ and $C^1=X_0R_0Y_0S_0Z_0T_0$. 
\end{center}  

\noindent It can be verified that $C^0$ and $C^1$ are of length $m=1+6k$ and are disjoint.  Hence $F_0=C^0\cup C^1$ is a $\vec{C}_m$-factor. 

 We proceed by constructing the following eight dipaths: 

\begin{multicols}{2}
\medskip
{\centering
  $ \displaystyle
    \begin{aligned} 
&W_1=x_0 x_1 x_2 y_3 x_4 y_5 y_{6} x_{7}  x_{8} y_{9} y_{10} x_{11} y_{12} x_{13};\\
&X_1=y_2 x_5 y_{8} y_{11} y_{14};\\
&Y_1=y_1 y_4 y_{7} x_{10} y_{13}; \\
&Z_1=y_0 x_3 x_{6} x_{9} x_{12} y_{15};
    \end{aligned}
  $ 
\par}
\medskip
\medskip
{\centering
  $ \displaystyle
    \begin{aligned} 
&Q_1=x_{13}x_{14} x_{15} \ldots x_{m-2} x_{m-1} x_0;\\
&R_1=y_{14} y_{17}y_{20} \ldots y_{m-5} y_{m-2}y_{1};\\
&S_1=y_{13}y_{16} y_{19} \ldots y_{m-6} y_{m-3} y_0; \\
&T_1=y_{15}y_{18} y_{21} \ldots y_{m-4}  y_{m-1} y_2.       
    \end{aligned}
  $ 
\par}
\medskip
\end{multicols}

We then form the following pair of directed cycles of length $m$:

\begin{center}
$C^2=W_1Q_1$ and $C^3=X_1R_1Y_1S_1Z_1T_1$. 
\end{center}  

It can be verified that $C^2$ and $C^3$ are disjoint. Therefore, we have a $\vec{C}_m$-factor $F_1=C^2\cup C^3$. 

\pagebreak
Next, form eight dipaths as follows:

\begin{multicols}{2}
\medskip
{\centering
  $ \displaystyle
    \begin{aligned} 
&W_2=x_0 y_1 x_2 x_3 y_4x_5 x_{6} x_{7} y_{8} y_{9} x_{10} x_{11} x_{12} x_{13}; \\
&X_2=y_2 y_5 x_{8} y_{11} x_{14};\\ 
&Y_2=x_1 x_4 y_{7} y_{10} y_{13};\\ 
&Z_2=y_0 y_3 y_{6} x_{9} y_{12} y_{15};
    \end{aligned}
  $ 
\par}
\medskip

\medskip
{\centering
  $ \displaystyle
    \begin{aligned} 
&Q_2= x_{13} y_{14} x_{15} y_{16} \ldots y_{m-3} x_{m-2} y_{m-1} x_0;\\
&R_2= x_{14} y_{17} x_{20} y_{23} \ldots x_{m-5} y_{m-2} x_1;\\ 
&S_2= y_{13}x_{16} y_{19} x_{22}\ldots y_{m-6} x_{m-3} y_0;\\ 
&T_2=y_{15} x_{18} y_{21} x_{24}\ldots y_{m-4} x_{m-1} y_2,
    \end{aligned}
  $ 
\par}
\medskip
\end{multicols}

\noindent and form the $\vec{C}_m$-factor $F_2=C^4\cup C^5$ where $C^4=W_2Q_2$ and $C^5=X_2R_2Y_2S_2Z_2T_2$. 

Finally, construct the following eight dipaths:

\begin{multicols}{2}
\medskip
{\centering
  $ \displaystyle
    \begin{aligned} 
&W_3=y_0x_1 y_2 y_3y_4y_5 x_{6} y_{7}y_{8}x_{9}y_{10} y_{11} y_{12} y_{13};\\
&X_3=x_2 x_5 x_{8}x_{11}y_{14};\\
&Y_3=y_1 x_4 x_{7}x_{10} x_{13};\\
&Z_3=x_0 x_3 y_{6}y_{9}x_{12} x_{15};
    \end{aligned}
  $ 
\par}
{\centering
  $ \displaystyle
    \begin{aligned} 
 &Q_3=y_{13}x_{14} y_{15} x_{16} \ldots y_{m-2} x_{m-1} y_0;\\
&R_3=y_{14}x_{17} y_{20} x_{23} \ldots y_{m-5} x_{m-2} y_{1};\\ 
&S_3=x_{13}y_{16} x_{19} y_{22} \ldots x_{m-6} y_{m-3} x_0;\\ 
&T_3=x_{15} y_{18} x_{21} y_{24} \ldots x_{m-4}  y_{m-1} x_2,
    \end{aligned}
  $ 
\par}
\medskip
\end{multicols}

\noindent and let $F_3=C^6\cup C^7$ for $C^6=W_3Q_3$ and $C^7=X_3R_3Y_3S_3Z_3T_3$, so $F_3$ is a $\vec{C}_m$-factor. 

It is laborious yet routine to verify that each arc of $L_{2m}$ appears exactly once in $\{F_0, F_1, $ $F_2,F_3\}$. Therefore, the set $\{F_0, F_1, F_2, F_3\}$ is a $\vec{C}_m$-factorization of $L_{2m}$. 

\noindent {\bf Case 2:} $p=5$. First, we construct a set of twelve dipaths. For each $i$, the dipaths $R'_i$ and $S'_i$ are of length $2(k-2)+1$, while the dipath $T'_i$ is of length $2(k-2)+2$:

\begin{multicols}{2}
\medskip
{\centering
  $ \displaystyle
    \begin{aligned} 
     &R'_0=x_{14} x_{17} x_{20} \ldots x_{m-6} x_{m-3} x_0;\\
&S'_0= x_{15}x_{18} x_{21} \ldots x_{m-5} x_{m-2}x_1;\\
&T'_0=x_{13}x_{16} x_{19} \ldots x_{m-4} x_{m-1} x_2;\\
&R'_1=y_{14}y_{17} y_{20} \ldots y_{m-6}  y_{m-3} y_{0};\\
 &S'_1=y_{15} y_{18} y_{21} \ldots y_{m-5}  y_{m-2} y_1;\\
  &T'_1=y_{13} y_{16} y_{19} \ldots y_{m-4}y_{m-1} y_2;\\
    \end{aligned}
  $ 
\par}
\medskip

\medskip
{\centering
  $ \displaystyle
    \begin{aligned} 
     &R'_2=x_{14} y_{17} x_{20} y_{23} \ldots y_{m-6} x_{m-3} y_0;\\
      &S'_2=y_{15} x_{18} y_{21} x_{24} \ldots x_{m-5} y_{m-2} x_1;\\ 
     &T'_2=y_{13} x_{16} y_{19} x_{22} \ldots y_{m-4} x_{m-1} y_2;\\
     &R'_3=y_{14} x_{17} y_{20} x_{23} \ldots x_{m-6} y_{m-3} x_0; \\
      &S'_3=x_{15} y_{18} x_{21} y_{24} \ldots y_{m-5} x_{m-2} y_{1}; \\
     &T'_3=x_{13}y_{16} x_{19} y_{22} \ldots x_{m-4} y_{m-1} x_2.\\
    \end{aligned}
  $ 
\par}
\medskip
\end{multicols}

We also use $Q_i, W_i, X_i, Y_i$, and $Z_i$ (for $i=0,\ldots , 3)$ from Case 1. Observe that, for each $i\in \{0,1,2,3\}$, the dipaths in $\{Q_i, R'_i, S'_i, T'_i, W_i, X_i, Y_i, Z_i\}$ are pairwise disjoint. For $i=0,\ldots, 3$, we now construct eight directed $m$-cycles:

\medskip
{\centering
  $ \displaystyle
    \begin{aligned} 
& C^i=W_iQ_i &\textrm{and} \ C^{i+4}=X_iR'_iZ_iS'_iY_iT'_i,  
    \end{aligned}
  $ 
\par}
\medskip

\noindent and it can be verified that $F_i=C^i \cup C^{i+4}$ is a $\vec{C}_m$-factor of $L_{2m}$. Furthermore, it can also be verified that each arc of $L_{2m}$ occurs exactly once in $\{F_0, F_1, F_2, F_3\}$. Consequently, the set $\{F_0, F_1, F_2, F_3\}$ is a $\vec{C}_m$-factorization of $L_{2m}$.  \end{proof}

\section{Decomposition of $\vec{X}(m, \{1,3\})\wr K^*_2$}

The purpose of this section is to construct a $\vec{C}_m$-factorization of $G_{2m}$. First, we introduce some notation. 

\begin{notation} \rm
\label{not:bas} 
Let $m=p+12k$ for some non-negative integer $k$ and $p\in \{11, 13, 17, 19\}$. We let  

\medskip
{\centering
  $ \displaystyle
    \begin{aligned} 
&V_0=\{x_0, x_1, \ldots x_{p-1}\}\cup \{y_0, y_1, \ldots, y_{p-1}\}\ \textrm{and}\\
&V_i=\{x_{p+12(i-1)}, x_{p+12(i-1)+1}, \ldots, x_{p+12{i-1}}\} \cup \{y_{p+12(i-1)}, y_{p+12(i-1)+1}, \ldots,y_{p+12{i-1}}\},
    \end{aligned}
  $ 
\par}

\noindent for $i=1,\ldots, k$. Then, we let $V(G_{2m})=\bigcup\limits_{i=0}^{k} V_i$. 
\end{notation}

Observe that $V_0$ contains $2p$ vertices while each $V_i$, for $i=1,2,\ldots, k$, contains $24$ vertices. Next, we define a function on $V(G_{2m})$. 

\begin{definition} \rm
Let $m=p+12k$ for some non-negative integer $k$ and $p\in \{11, 13, 15, 17\}$. Define a function $\rho:V(G_{2m}) \rightarrow V(G_{2m})$ as follows: $\rho(x_i)=x_{i+1}$ and $\rho(y_i)=y_{i+1}$, with subscripts modulo $m$. 
\end{definition}

The aim of Lemmas \ref{lem:list2}, \ref{lem:list1}, and \ref{lem:list3} is to reduce the question of existence of a  $\vec{C}_{m}$-factorization of $G_{2m}$ to the existence of a particular set of 32 dipaths. The properties that these dipaths must satisfy are made precise in Definitions \ref{defn:basic2} and \ref{defn:basic1}. 

\begin{definition} \rm
\label{defn:basic2}
Let $m=p+12k$ for some non-negative integer $k$ and $p\in \{11, 13, 17, 19\}$. Let $\{W, X, Y, Z, Q,R, S,T\}$ be a set of dipaths of $G_{2m}$. The 8-tuple  $(W, X, Y, Z, Q, R, S, T)$ is a  \textit{type-2 basic set of dipaths of $G_{2m}$} if it satisfies the following properties.
\begin{enumerate} [label=\textbf{C\arabic*.}]
\item Dipaths in  $\{Q, R, S, T\}$ are pairwise disjoint. If $k\geqslant 1$ then  dipaths in  $\{W, X, Y, Z\}$ are pairwise disjoint; otherwise $WX$ and $YZ$ are disjoint type-2 directed cycles. 
\item $s(X)=\rho^{-p}(t(W))$, $s(W)=\rho^{-p}(t(X))$, $s(Z)=\rho^{-p}(t(Y))$, and $s(Y)=\rho^{-p}(t(Z))$.
\item $\len(W)+\len(X)=\len(Y)+\len(Z)=p$; $\len(Q)+\len(R)= \len(S)+\len(T)=12$ if $k \geqslant 1$, and $\len(Q)=\len(R)= \len(S)=\len(T)=0$ otherwise.
\item Each of $W$, $X$, $Y$, and $Z$ has its source and internal vertices in $V_0$, and its terminus $x_t$ or $y_t$ for some $t \in \{p, p+1, p+2\}$. 
\item $t(W)=s(Q), t(X)=s(R), t(Y)=s(S)$, and $t(Z)=s(T)$. 
\item If $k\geqslant 1$ and $P \in \{Q, R, S, T\}$ such that $s(P)=x_t$ (or $y_t$), then $t(P)=x_{t+12}$ (or $y_{t+12}$ respectively). Moreover, all internal vertices of $P$ are in $V_1$. 

\end{enumerate}
\end{definition}

\begin{lemma}
\label{lem:list2}
Let $m=p+12k$ for some non-negative integer $k$ and $p\in \{11, 13, 17, 19\}$. Suppose that $(W, X, Y, Z, P_0^0, P_0^1, Q_0^0, Q_0^1)$ is a type-2 basic set of dipaths of $G_{2m}$. For each $i \in \{0,1\}$ and $j \in \{1,2, \ldots, k-1\}$, let $P_j^i=\rho^{12j}(P^i_0) $ and $Q^i_j=\rho^{12j}(Q^i_0)$. Then
\begin{center}
$C^0=WP_0^0P^0_1\ldots P^0_{k-1}XP^1_0P_1^1\ldots P^1_{k-1}$ and $C^1=YQ_0^0Q_1^0\ldots Q_{k-1}^0ZQ_0^1Q_1^1\ldots Q_{k-1}^1$
\end{center}

\noindent  are type-2 directed $m$-cycles, and $C^0 \cup C^1$ is a $\vec{C}_m$-factor of $G_{2m}$.  
\end{lemma}

\begin{proof}
First, consider the case $k=0$. Observe that $V(G_{2m})=V_0$  and by property C3 of Definition \ref{defn:basic2}, dipaths $P_0^0, P_0^1, Q_0^0, Q_0^1$ are of length 0. Properties C1 and C3 jointly imply that $WX \cup YZ$ is a $\vec{C}_{m}$-factor of $G_{2m}$ consisting of two type-2 directed $m$-cycles.  

Assume that $k\geqslant 1$. Without loss of generality, by properties C4-C6 we may assume that  $P^0_0$ is an $(x_t, x_{t+12})$-dipath for some $t\in \{p, p+1, p+2\}$ and all internal vertices in $V_1$. Hence, the dipath $P^0_j$ is an $(x_{t+12j}, x_{t+12(j+1)})$-dipath and all of its internal vertices are in $V_j$. Therefore, for all $0\leqslant i < j \leqslant k-1$, dipaths in $P_i^0$ and $P_j^0$ are disjoint if $|i-j|>1$ and share only vertex $x_{t+12j}$ if $j=i+1$. Therefore, the concatenation 

\begin{center}
$I_0=P_0^0P^0_1\ldots P^0_{k-1}$.
\end{center}

\noindent is a well-defined ($x_t, x_{t+12k}$)-dipath.

An analogous observation holds for $P_0^1$ and dipath 

\begin{center}
$I_1=P^1_0P_1^1\ldots P^1_{k-1}$.
\end{center}

By C6, since $WP^0_0$ is a well-defined dipath, vertex $x_t$ is the terminus of $W$, and hence by C2, vertex $x_{t-p}=x_{t+12k}$ is the source of $X$. As a result, the concatenation $WI_0X$ is possible. Analogously, we show that $XI_1W$ is well-defined. 

Since $W$ and $X$, as well as $P^0_0$ and $P^1_0$, are disjoint by C1, and the sets of internal vertices of $W, I_0, X$, and $I_1$ are pairwise disjoint, it follows that $C^0=WI_0XI_1$ is a directed cycle.  Since $\len(W)+\len(X)=p$ and $\len(P^0_0)+\len(P^1_0)=12$ by C3, it follows that $C^0$ is a directed $m$-cycle. 

Analogously, we show that $C^1$ is a directed $m$-cycle. Property C1 then implies that $C^0$ and $C^1$ are disjoint. Therefore, the digraph $C^0 \cup C^1$ is a $\vec{C}_m$-factor of $G_{2m}$. 

Lastly, we see that arcs in $I_0\cup I_1$ sum to $24k$ and that arcs in $W \cup X$ sum to $2p$. It follows that $C^0$ is a type-2 directed $m$-cycles. Similarly, we can show that $C^1$ is also a type-2 directed $m$-cycle. \end{proof}

\begin{definition} \rm
\label{defn:basic1}
Let $m=p+12k$ for some  non-negative integer $k$ and $p\in \{11, 13, 17, 19\}$. Let the set $\{X, Y, R, S\}$ be a set of dipaths or directed cycles of $G_{2m}$. The 4-tuple  $(X, Y, R,S)$ is called a \textit{type-1 basic set of  dipaths} if it satisfies the following properties.
\begin{enumerate}[label=\textbf{C\arabic*.}]

\item Dipaths $R$ and $S$ are disjoint. If $k\geqslant 1$, then $X$ are $Y$ are disjoint dipaths; otherwise, they are disjoint type-1 directed cycles. 
\item $s(X)=\rho^{-p}(t(X))$ and $s(Y)=\rho^{-p}(t(Y))$.
\item $\len(X)=\len(Y)=p$;  $\len(R)=\len(S)=12$ if $k\geqslant 1$ and $\len(R)=\len(S)=0$ if $k=0$.
\item  Each of $X$ and $Y$ has its source and internal vertices in $V_0$, and its terminus is $x_t$ or $y_t$ for some $t \in \{p, p+1, p+2\}$. 
\item $t(X)=s(R)$ and $t(Y)=s(S)$.
\item If $k\geqslant 1$, and $P\in \{R, S\}$ such that $s(P)=x_t$ (or $y_t$), then $t(P)=x_{t+12}$ (or $y_{t+12}$, respectively). Moreover, all internal vertices of $P$ are in $V_1$. 
\end{enumerate}
\end{definition}

\begin{lemma}
\label{lem:list1}
Let $m=p+12k$ for some non-negative integer $k$ and $p\in \{11, 13, 17, 19\}$. Suppose that $(X, Y, P_0,Q_0)$ is a type-1 basic set of dipaths of $G_{2m}$. For each $j \in \{1,2, \ldots, k-1\}$, let $P_j=\rho^{12j}(P_0)$ and $Q_j=\rho^{12j}(Q_0)$. Then
\begin{center}
$C^0=XP_0P_1\ldots P_{k-1}$ and  $C^1=YQ_0Q_1\ldots Q_{k-1}$
\end{center}

\noindent are type-1 directed $m$-cycles, and $C^0\cup C^1$ is a $\vec{C}_m$-factor of $G_{2m}$.
\end{lemma}

\begin{proof}
First, assume that $k=0$. Then, properties C1 and C3 of Definition \ref{defn:basic1} jointly imply that $X \cup Y$ is a $\vec{C}_m$-factor of $G_{2m}$ consisting of two type-1 directed $m$-cycles. 

Without loss of generality, by properties C4-C6, we may  assume that $P_0^0$ is an $(x_t, x_{t+12})$-dipath for some $t\in \{p, p+1, p+2\}$ and all of its internal vertices are in $V_1$. Hence $P^j_0$ is an $(x_{t+12j}, x_{t+12(j+1)})$-dipath with all of its internal vertices in $V_j$. As a result, for all $0 \leqslant i < j \leqslant k-1$, dipaths $P_0^i$ and $P_0^j$ are disjoint if $|i-j|>1$, and share only vertex $x_{t+12j}$ if $j=i+1$. Therefore, the concatenation
\begin{center}
$I_0=P^0_0P^1_0\ldots P^0_{k-1}$
\end{center}

\noindent is a well-defined $(x_t, x_{t+12k})$ dipath. 

Similarly, we can construct dipath $I_1$ as follows:
 
 \begin{center}
$I_1=Q^0_0Q^1_0\ldots Q^0_{k-1}$.
\end{center}

Next, by C5, it follows that $XI_0$ is a well-defined concatenation. Furthermore, by C2 vertex $x_{t-p}=x_{t+12k}$ is also the source of $X$. Therefore, the concatenation $ I_0X$ is also well-defined. Since the sets of internal vertices of $X$ and $I_0$ are disjoint, it follows that $C^0=XI_0$ is a directed cycle. Since $\len(X)=p$ and $\len(I_0)=12k$ by C3, the directed cycle $C^0$ is of length $m$. 

Analogously, it can be shown that $C^1$ is also a directed $m$-cycle. Property C1 then implies that $C^0$ and $C^1$ are disjoint. Therefore, the digraph $C^0 \cup C^1$ is a $\vec{C}_m$-factor of $G_{2m}$. 

Lastly, we see that arcs in $I_0$ and $I_1$ sum to $12k$, respectively. Moreover, arcs in $X$ and $Y$ sum to $p$, respectively. It follows  that $C^0$ and $C^1$ are both type-1 directed $m$-cycles.  \end{proof}

\begin{lemma}
\label{lem:list3}
Let $m=p+12k$ with $k$ a non-negative integer and $p \in \{11, 13, 17, 19\}$. If $G_{2m}$ admits three type-2 basic sets of dipaths and two type-1 basic sets of dipaths such that the dipaths and directed cycles in these five sets are pairwise arc-disjoint, then $G_{2m}$ admits a $\vec{C}_m$-factorization 
\end{lemma}

\begin{proof}
Let $A_1, A_2$, and $A_3$ be type-2 basic sets of dipaths, and $A_4$ and $A_5$ be type-1 basic sets of dipaths such that the 32 dipaths and directed cycles in $S=A_1\cup A_2 \cup A_3 \cup A_4 \cup A_5$ are pairwise arc-disjoint.  

By Lemma \ref{lem:list2}, each $A_i$ for $i \in \{1,2,3\}$ gives rise to a $\vec{C}_m$-factor $F_i$ of $G_{2m}$ consisting of two type-2 directed $m$-cycles. In addition, by Lemma \ref{lem:list1}, each $A_i$ for $i \in \{4,5\}$ gives rise to a $\vec{C}_m$-factor $F_i$ consisting of  two type-1 directed $m$-cycles. It remains to show that the $F_i$, for $i \in \{1,2,\ldots, 5\}$,  are pairwise arc-disjoint. 

Suppose that an arc $a=(x_r, x_s)$ of $G_{2m}$, where $x_s\in V_j$ for some $j\in \{0, \ldots, k\}$, occurs twice in the $\vec{C}_m$-factors $F_1, \ldots, F_5$. By the construction of $F_1, \ldots, F_5$ from Lemmas \ref{lem:list2} and \ref{lem:list1}, it follows that $j\geqslant 2$, and the arc $a'=(x_{r-12(j-1)}, x_{s-12(j-1)})$ also occurs twice in $F_1, \ldots, F_5$. However, we see that $a'$ has a tail in $V_1$, and thus appears in $S$. By assumption, the arc $a'$ appears exactly once in $S$, and by construction from Lemmas \ref{lem:list2} and \ref{lem:list1}, it appears nowhere else in $F_1, \ldots, F_5$, a contradiction.  

An analogous argument applies if $a$ is of the form $(y_r, y_s)$, $(y_r, x_s)$, and $(x_r, y_s)$. Therefore, each arc of $G_{2m}$ occurs at most once in $F_1,\ldots, F_5$. Since $G_{2m}$ is of degree five, it follows that every arc occurs exactly once. Therefore, the set $\{F_1, \ldots, F_5\}$ is a $\vec{C}_m$-factorization of $G_{2m}$.   \end{proof}

Lemma \ref{lem:list3} implies that, to construct a $\vec{C}_m$-factorization of $G_{2m}$, it suffices to build three type-2 basic sets of dipaths and two type-1 basic sets of dipaths such that the dipaths and directed cycles in the union of these five sets are pairwise arc-disjoint.  

In the proof of Proposition \ref{thm:cand} below, we consider four cases, one for each $p\in \{11, 13, 17,$ $19\}$. This yields a construction for each congruency class modulo 12 for odd $m \geqslant 11$ and $3 \nmid m$. In each case, we construct three type-2 basic sets of dipaths and two type-1 basic sets. It is straightforward, though tedious, to verify that the hypotheses of Lemma \ref{lem:list3} are satisfied by these five sets of dipaths. To aid in the verification of these properties, we illustrate all dipaths constructed in the proof of Proposition \ref{thm:cand} in Appendix A. 

\begin{prop}
\label{thm:cand}
Let $m\geqslant 11$ be an odd integer. The digraph $G_{2m}$ admits a $\vec{C}_m$-factorization. 
\end{prop}

\begin{proof}
Let $m=p+12k$ with $k$ a non-negative integer and $p \in \{11,13,17,19\}$. Throughout this proof, we shall refer to Notation \ref{not:2di} and \ref{not:bas}. In each case, we construct three type-2 basic sets of dipaths $L_0, L_1$, and $L_2$, and two type-1 basic sets of dipaths $L_3$ and $L_4$. In each case, it can then be verified that dipaths in $L_0\cup L_1\cup L_2 \cup L_3 \cup L_4$ are pairwise arc-disjoint, thereby satisfying the hypotheses of Lemma \ref{lem:list3}. 

\noindent {\bf Case 1:} $p=11$. See Figures \ref{fig:m11.1.1}-\ref{fig:m11.5.2}.

Let $L_0=(W_0, X_0, Y_0, Z_0, Q_0, R_0, S_0, T_0)$ where

\begin{multicols}{2}
{\centering
  $ \displaystyle
    \begin{aligned} 
&W_0=x_0 x_3 x_4 y_7 x_7 y_8 y_{11};\\
&X_0=y_0 y_1 y_4 y_5 x_8 x_{11};\\
&Y_0=x_1 y_2 x_5 x_6 x_9 x_{10} y_{10} x_{13};\\
&Z_0=x_2 y_3 y_6 y_9 x_{12};
    \end{aligned}
  $ 
\par}
{\centering
  $ \displaystyle
    \begin{aligned} 
&Q_0=y_{11} y_{14} y_{17} y_{18} y_{19} x_{22} y_{23};\\
&R_0=x_{11} y_{12} y_{13} x_{16} x_{17} y_{20} x_{23};\\
&S_0=x_{13} x_{14} x_{15} y_{16} x_{19} x_{20} x_{21} y_{22} x_{25};\\
&T_0=x_{12} y_{15} x_{18} y_{21} x_{24}.
    \end{aligned}
  $ 
\par}
\end{multicols}

Let $L_1=(W_1, X_1, Y_1,  Z_1, Q_1, R_1, S_1, T_1)$ where

\begin{multicols}{2}
{\centering
  $ \displaystyle
    \begin{aligned} 
&W_1=x_0 y_1 x_4 y_5 y_8 y_9 x_9 x_{12};\\
&X_1=x_1 y_4x_7 y_{10} x_{11};\\
&Y_1=y_0 y_3x_3 x_6 y_6 y_7 x_{10} y_{13};\\
&Z_1=y_2 x_2 x_5 x_8 y_{11};
    \end{aligned}
  $ 
\par}
{\centering
  $ \displaystyle
    \begin{aligned} 
&Q_1=x_{12} y_{12} y_{15} x_{15} x_{18} y_{18} y_{21} x_{21}  x_{24}  ;\\
&R_1=x_{11} y_{14} x_{17} x_{20} x_{23};\\
&S_1=y_{13} x_{13} y_{16} x_{16} y_{19} x_{19} y_{22} x_{22} y_{25};\\
&T_1=y_{11} x_{14} y_{17} y_{20} y_{23}.
    \end{aligned}
  $ 
\par}
\end{multicols}

Let $L_2=(W_2, X_2, Y_2, Z_2, Q_2, R_2, S_2, T_2)$ where

\begin{multicols}{2}
{\centering
  $ \displaystyle
    \begin{aligned} 
&W_2=x_0 y_3 x_6 y_9 y_{10}y_{11};\\
&X_2=y_0 x_3 y_4 y_7 x_8 y_8 x_{11};\\
&Y_2=y_1 x_1 x_4 x_7 x_{10} x_{13};\\
&Z_2=x_2 y_2 y_5 x_5 y_6 x_9 y_{12};
\medskip
    \end{aligned}
  $ 
\par}
{\centering
  $ \displaystyle
    \begin{aligned} 
&Q_2=y_{11} x_{12} y_{13} y_{16} y_{17} x_{20} y_{23};\\
&R_2=x_{11} x_{14} x_{17} x_{18} y_{19} y_{22} x_{23};\\
&S_2=x_{13} y_{14} y_{15} x_{16} x_{19} y_{20} y_{21} x_{22} x_{25};\\
&T_2=y_{12} x_{15} y_{18} x_{21} y_{24}.
    \end{aligned}
  $ 
\par}
\end{multicols}

For $k=0$, we replace each dipath in $\{Q_i, R_i, S_i, T_i\ |\ i=0,1,2\}$ with a dipath of length $0$ with the same source. It can  be verified that the 8-tuples $L_0$, $L_1$, and $L_2$ satisfy conditions C1-C6 of Definition \ref{defn:basic2}, and thus are type-2 basic sets of dipaths. 

Let  $L_3=(X_3, Y_3, R_3, S_3)$ where

\begin{multicols}{2}
{\centering
  $ \displaystyle
    \begin{aligned} 
&X_3=x_0 y_0 x_1y_1 x_2 x_3 y_6 x_7x_8y_9 x_{10}x_{11};\\
&Y_3=y_2 y_3 x_4y_4x_5y_5 x_6 y_7 y_8 x_9 y_{10} y_{13};\\
\medskip
    \end{aligned}
  $ 
\par}
{\centering
  $ \displaystyle
    \begin{aligned} 
&\hspace{-0.6cm} R_3=x_{11} y_{11} y_{12} x_{12} x_{13} x_{16} y_{17} x_{17} y_{18} x_{18} x_{19} x_{22} x_{23};\\
&\hspace{-0.6cm} S_3=y_{13} x_{14} y_{14} x_{15} y_{15} y_{16} y_{19} x_{20} y_{20} x_{21}y_{21} y_{22} y_{25}.\\
    \end{aligned}
  $ 
\par}
\end{multicols}
\vspace{-8mm}

\pagebreak
Finally, let $L_4=(X_4, Y_4, R_4, S_4)$ where
\begin{multicols}{2}
{\centering
  $ \displaystyle
    \begin{aligned} 
&X_4=y_0 x_0 x_1 x_2y_5 y_6 x_6x_7 y_7 y_{10} x_{10} y_{11};\\
&Y_4=y_1 y_2 x_3 y_3 y_4 x_4 x_5 y_8 x_8 x_9 y_9 y_{12}; \\
    \end{aligned}
  $ 
\par}
{\centering
  $ \displaystyle
    \begin{aligned} 
&\hspace{-0.6cm} R_4=y_{11} x_{11} x_{12} x_{15} x_{16} y_{16} x_{17}y_{17} x_{18} x_{21} x_{22} y_{22} y_{23};\\
&\hspace{-0.6cm} S_4=y_{12} x_{13} y_{13}y_{14} x_{14}y_{15}y_{18}x_{19} y_{19} y_{20} x_{20}y_{21} y_{24}.\\
    \end{aligned}
  $ 
\par}
\end{multicols}

Again, for $k=0$, we replace each dipath in $\{R_i, S_i \ |\ i=3, 4\}$ with a dipath of length $0$ with the same source. Similarly, it can be verified that quadruples $L_3$ and $L_4$ satisfy conditions C1-C6 of Definition \ref{defn:basic1}, and thus are type-1 basic sets of dipaths.  

It is tedious, but straightforward, to verify that, for each vertex $x_i$ ($y_i$) in $V_0$, each of the five arcs with tail $x_i$ ($y_i$) appears in a dipath in $\{W_j, X_j, Y_j, Z_j: j=0,1,2\} \cup \{X_j, Y_j: j=3,4\}$ exactly once. An analogous claim holds for vertices $x_i$ ($y_i$) in $V_1$ and the set of dipaths $\{Q_j, R_j, S_j, T_j: j=0,1,2\}\cup \{R_j, S_j: j=3,4\}$. It follows that the dipaths in $L_0\cup \ldots \cup L_4$ are pairwise arc-disjoint and thus satisfy the hypothesis of Lemma \ref{lem:list3}. As a result, the digraph $G_{2m}$ admits a $\vec{C}_m$-factorization.

\noindent {\bf Case 2:} $p=13$. See Figures \ref{fig:m13.1.1}-\ref{fig:m13.5.2}.

Let $L_0=(W_0, X_0, Y_0, Z_0, Q_0, R_0, S_0, T_0)$ where

\begin{multicols}{2}
{\centering
  $ \displaystyle
    \begin{aligned} 
&W_0=y_0 x_1 x_2 y_5 x_6 x_9 y_{12} x_{13};\\
&X_0=x_0 y_3 y_4 y_7 y_{10} x_{10} y_{13};\\
&Y_0=y_1 x_4 x_5 x_8 y_9 x_{12} y_{15};\\
&Z_0=y_2 x_3 y_6 x_7 y_8 x_{11} y_{11} y_{14};\\
    \end{aligned}
  $ 
\par}
{\centering
  $ \displaystyle
    \begin{aligned} 
&Q_0=x_{13} y_{16} y_{19} x_{20} x_{21} y_{24} x_{25}; \\
&R_0= y_{13} x_{14} x_{15} y_{18} x_{19} y_{22} y_{25};\\
&S_0=y_{15} x_{16} y_{17} x_{18} y_{21} x_{22} y_{23} x_{24} y_{27};\\
&T_0=y_{14} x_{17} y_{20} x_{23} y_{26}.
    \end{aligned}
  $ 
\par}
\end{multicols}

Let $L_1=(W_1, X_1, Y_1, Z_1, Q_1, R_1, S_1, T_1)$ where

\begin{multicols}{2}
{\centering
  $ \displaystyle
    \begin{aligned} 
&W_1=x_0 y_0 x_3 x_6 x_7 y_{10}y_{11} x_{14};\\
&X_1=x_1 y_4 x_5 y_6 y_{9} x_{10} x_{13};\\
&Y_1=y_2 y_5 x_8 x_{11} y_{14};\\
&Z_1=y_1 x_2 y_3 x_4 y_7 y_8 x_9 x_{12} y_{12}y_{15};\\
    \end{aligned}
  $ 
\par}
{\centering
  $ \displaystyle
    \begin{aligned} 
&Q_1=x_{14} x_{17} x_{20} x_{23} x_{26};\\
&R_1= x_{13} y_{13} y_{16} x_{16} x_{19} y_{19} y_{22} x_{22} x_{25};\\
&S_1=y_{14} y_{17} y_{20} y_{23} y_{26};\\
&T_1=y_{15} x_{15} x_{18} y_{18} y_{21} x_{21} x_{24} y_{24} y_{27};\\
    \end{aligned}
  $ 
\par}
\end{multicols}

Let $L_2=(W_2, X_2, Y_2, Z_2, Q_2, R_2, S_2, T_2)$ where

\begin{multicols}{2}
{\centering
  $ \displaystyle
    \begin{aligned} 
&W_2=x_0 x_3 y_3  y_6 x_9 y_{10} y_{13};\\
&X_2=y_0 y_1 y_4 x_7 x_8 y_{11} x_{12} x_{13};\\
&Y_2=x_2 y_2 x_5 x_6 y_7 x_{10} x_{11}  x_{14};\\
&Z_2=x_1 x_4 y_5 y_8 y_9 y_{12} x_{15};\\
    \end{aligned}
  $ 
\par}
{\centering
  $ \displaystyle
    \begin{aligned} 
&Q_2=y_{13}y_{14} y_{15} x_{18} x_{19} x_{22} y_{25};\\
&R_2=x_{13} x_{16} y_{19} y_{20} y_{21} x_{24} x_{25};\\
&S_2=x_{14} y_{17} x_{20} y_{23} x_{26} ;\\
&T_2=x_{15} y_{16} x_{17} y_{18} x_{21} y_{22} x_{23} y_{24} x_{27}.
    \end{aligned}
  $ 
\par}
\end{multicols}

For $k=0$, we replace each dipath in $\{Q_i, R_i, S_i, T_i\ |\ i=0,1,2\}$ with a dipath of length $0$ with the same source. It can  be verified that the 8-tuples $L_0$, $L_1$, and $L_2$ satisfy conditions C1-C6 of Definition \ref{defn:basic2}, and thus are type-2 basic sets of dipaths. 

\pagebreak

Let $L_3=(X_3, Y_3, R_3, S_3)$ where

\begin{multicols}{2}
{\centering
  $ \displaystyle
    \begin{aligned} 
&X_3=y_0 x_0 y_1 x_1 y_2 y_3 x_6 y_6 y_7 x_7 x_{10} y_{11} y_{12} y_{13};\\
&Y_3=x_2 x_3 x_4 y_4 y_5 x_5 y_8 x_8 x_9 y_9 y_{10} x_{11} x_{12}  x_{15};\\
    \end{aligned}
  $ 
\par}
{\centering
  $ \displaystyle
    \begin{aligned} 
&R_3= y_{13} x_{13} y_{14} x_{14} y_{15} y_{18} y_{19} x_{19} y_{20} x_{20} y_{21} y_{24} y_{25};\\
&S_3=x_{15} x_{16} y_{16} y_{17} x_{17} x_{18} x_{21} x_{22}y_{22} y_{23} x_{23} x_{24} x_{27}.\\
    \end{aligned}
  $ 
\par}
\end{multicols}

Let $L_4=(X_4, Y_4, R_4, S_4)$ where

\begin{multicols}{2}
{\centering
  $ \displaystyle
    \begin{aligned} 
&X_4=x_0 x_1 y_1 y_2 x_2 x_5 y_5y_6 x_6 y_9 x_9 x_{10} y_{10} x_{13};\\
&Y_4=y_0 y_3 x_3 y_4 x_4x_7 y_7 x_8 y_8 y_{11}x_{11} y_{12} x_{12}  y_{13};\\
    \end{aligned}
  $ 
\par}
{\centering
  $ \displaystyle
    \begin{aligned} 
&R_4= x_{13}x_{14} y_{14} x_{15} y_{15} y_{16} x_{19} x_{20} y_{20} x_{21} y_{21} y_{22} x_{25};\\
&S_4=y_{13} x_{16} x_{17} y_{17} y_{18} x_{18} y_{19} x_{22} x_{23} y_{23}y_{24} x_{24} y_{25}.\\
    \end{aligned}
  $ 
\par}
\end{multicols}

Again, for $k=0$, we replace each dipath in $\{R_i, S_i \ |\ i=3, 4\}$ with a dipath of length $0$ with the same source. Similarly, it can be verified that quadruples $L_3$ and $L_4$ satisfy conditions C1-C6 of Definition \ref{defn:basic1}, and thus are type-1 basic sets of dipaths. 

Once again, it is tedious but straightforward to verify that $\{L_0, L_1, L_2, L_3, L_4\}$ satisfies the hypothesis of Lemma \ref{lem:list3}.

\noindent {\bf Case 3:} $p=17$. See Figures \ref{fig:m17.1.1}-\ref{fig:m17.5.2}.

Let $L_0=(W_0, X_0, Y_0,  Z_0, Q_0, R_0, S_0, T_0)$ where

\begin{multicols}{2}

{\centering
  $ \displaystyle
    \begin{aligned} 
&W_0=y_0 x_0y_3x_6y_7x_8 x_{11} y_{12}x_{15}x_{18};\\
&X_0=x_1 y_2 x_5 y_6 y_9x_{10}x_{13}y_{14}y_{17} ;\\
&Y_0=y_1 y_4y_5y_8 x_9 x_{12}y_{13} y_{16}x_{16}x_{19};\\
&Z_0=x_2 x_3 x_4x_7 y_{10} y_{11}x_{14} y_{15} y_{18};
    \end{aligned}
  $ 
\par}
{\centering
  $ \displaystyle
    \begin{aligned} 
&Q_0=x_{18} y_{21} x_{24} y_{27} x_{30};\\
&R_0= y_{17} x_{17}y_{20} x_{20} y_{23} x_{23} y_{26} x_{26} y_{29};\\
&S_0=x_{19}y_{19} x_{22}y_{22} x_{25} y_{25} x_{28}y_{28} x_{31};\\
&T_0=y_{18} x_{21}y_{24} x_{27} y_{30}.
    \end{aligned}
  $ 
\par}
\medskip

\end{multicols}

Let $L_1=(W_1, X_1, Y_1, Z_1, Q_1, R_1, S_1, T_1)$ where

\begin{multicols}{2}
\medskip
{\centering
  $ \displaystyle
    \begin{aligned} 
&W_1=x_1 x_4y_4 x_7 y_7 x_{10} y_{10} y_{13} y_{14} x_{17};\\
&X_1=x_0 x_3 y_3 y_6 x_6 x_9y_{12} y_{15}x_{18} ;\\
&Y_1=y_0 y_1 y_2 y_5 x_8 y_{11} x_{11} x_{14} x_{15} y_{16} x_{19};\\
&Z_1=x_2  x_5 y_8 y_9 x_{12} x_{13} x_{16} y_{17};
    \end{aligned}
  $ 
\par}
\medskip
{\centering
  $ \displaystyle
    \begin{aligned} 
&Q_1=x_{17} y_{18} y_{19} y_{22} x_{23} y_{24} y_{25} y_{28}x_{29};\\
&R_1= x_{18} x_{21} x_{24} x_{27} x_{30};\\
&S_1=x_{19} x_{20} y_{21} x_{22} x_{25} x_{26} y_{27} x_{28} x_{31};\\
&T_1=y_{17} y_{20} y_{23} y_{26} y_{29}.
    \end{aligned}
  $ 
\par}
\end{multicols}

Let $L_2=(W_2, X_2, Y_2, Z_2, Q_2, R_2, S_2, T_2)$ where

\begin{multicols}{2}
{\centering
  $ \displaystyle
    \begin{aligned} 
&W_2=y_0 x_3 x_6 x_7 y_8y_{11}y_{14}x_{14} x_{17}; \\
&X_2=x_0 x_1 y_4 x_5 x_8 y_9y_{12} x_{13} y_{16} y_{17};\\ 
&Y_2=y_1 x_2 y_5 y_6x_9x_{10} y_{13} x_{16} y_{19} ; \\
&Z_2= y_2 y_3 x_4 y_7y_{10} x_{11} x_{12} y_{15}x_{15} y_{18};\\
    \end{aligned}
  $ 
\par}
{\centering
  $ \displaystyle
    \begin{aligned} 
 &Q_2=x_{17} x_{18}x_{19}y_{22} y_{23} x_{26} x_{29};\\
&R_2= y_{17}x_{20} x_{23}x_{24} x_{25}y_{28} y_{29};\\
&S_2=y_{19}y_{20} x_{21} x_{22} y_{25}y_{26}x_{27} x_{28} y_{31};\\
 &T_2=y_{18}y_{21} y_{24} y_{27} y_{30}.
    \end{aligned}
  $ 
\par}
\end{multicols}

For $k=0$, we replace each dipath in $\{Q_i, R_i, S_i, T_i\ |\ i=0,1,2\}$ with a dipath of length $0$ with the same source. It can  be verified that the 8-tuples $L_0$, $L_1$, and $L_2$ satisfy conditions C1-C6 of Definition \ref{defn:basic2}, and thus are type-2 basic sets of dipaths. 

\pagebreak
Let $L_3=(X_3, Y_3, R_3, S_3)$ where

\medskip
{\centering
  $ \displaystyle
    \begin{aligned} 
&X_3=x_0 y_0 x_1 y_1 x_4 y_5 x_5 x_6 y_6 x_7 x_{10} x_{11} y_{11} y_{12} x_{12} x_{15} x_{16} x_{17};\\
&Y_3=y_2 x_2 y_3 x_3 y_4 y_7 y_8 x_8 x_9 y_9y_{10} x_{13} y_{13} x_{14}y_{14} y_{15} y_{16} y_{19} ;\\
&R_3= x_{17} y_{17} x_{18} y_{18} x_{19} x_{22} x_{23} y_{23} x_{24} y_{24} x_{25} x_{28} x_{29} ;\\
&S_3=y_{19} x_{20} y_{20} y_{21} x_{21}y_{22} y_{25} x_{26}y_{26} y_{27}x_{27}y_{28} y_{31}.\\
    \end{aligned}
  $ 
\par}
 \medskip
 
Let $L_4=(X_4, Y_4, R_4, S_4)$ where

\medskip
{\centering
  $ \displaystyle
    \begin{aligned} 
&X_4=x_0 y_1 x_1x_2y_2 x_3 y_6 y_7 x_7 x_8 y_8 x_{11} y_{14} x_{15} y_{15} x_{16} y_{16} x_{17};\\
&Y_4=y_0 y_3 y_4 x_4 x_5 y_5 x_6 y_9 x_9 y_{10} x_{10} y_{11} x_{12} y_{12} y_{13} x_{13} x_{14}y_{17};\\
&R_4= x_{17} x_{20} x_{21} y_{21} y_{22} x_{22} y_{23} y_{24} x_{24} y_{25} x_{25} y_{26} x_{29};\\
&S_4=y_{17} y_{18} x_{18}y_{19} x_{19} y_{20} x_{23} x_{26} x_{27} y_{27}y_{28} x_{28} y_{29}.\\
    \end{aligned}
  $ 
\par}
\medskip
 
Again, for  $k=0$, we replace each dipath in $\{R_i, S_i \ |\ i=3, 4\}$ with a dipath of length $0$ with the same source. Similarly, it can be verified that quadruples $L_3$ and $L_4$ satisfy conditions C1-C6 of Definition \ref{defn:basic1}, and thus are type-1 basic sets of dipaths. 
 
Similarly to Cases 1 and 2,  it can be verified that $\{L_0, L_1, L_2, L_3, L_4\}$ satisfies the hypothesis of Lemma \ref{lem:list3}. 
 
\noindent {\bf Case 4:} $p=19$. See Figures \ref{fig:m19.1.1}-\ref{fig:m19.5.2}.

Let $L_0=(W_0, X_0, Y_0, Z_0, Q_0, R_0, S_0, T_0)$ where

\begin{multicols}{2}
\medskip
{\centering
  $ \displaystyle
    \begin{aligned} 
&W_0=y_0y_1 x_4 y_4 y_7 x_{10} x_{13} y_{14} y_{17} x_{20};\\
&X_0=x_1 y_2 y_5 y_6 x_7 y_{10} x_{11} y_{12} y_{13} y_{16} y_{19} ;\\
&Y_0=x_0 y_3 x_3 x_6 y_9 x_9 x_{12} y_{15} x_{18} x_{21};\\
&Z_0=x_2 x_5 y_8 x_8 y_{11} x_{14} x_{15} x_{16} x_{17} y_{18} x_{19};\\
    \end{aligned}
  $ 
\par}
\medskip
{\centering
  $ \displaystyle
    \begin{aligned} 
&Q_0=x_{20} y_{21} x_{22} y_{23} x_{26} y_{27} x_{28} y_{29} x_{32};\\
&R_0= y_{19} y_{22} y_{25} y_{28} y_{31};\\
&S_0=x_{21} x_{24} x_{27} x_{30} x_{33};\\
&T_0=x_{19} y_{20} x_{23} y_{24} x_{25} y_{26} x_{29} y_{30} x_{31}.
    \end{aligned}
  $ 
\par}
\medskip
\end{multicols}

Let $L_1=(W_1, X_1, Y_1,  Z_1, Q_1, R_1, S_1, T_1)$ where

\begin{multicols}{2}
{\centering
  $ \displaystyle
    \begin{aligned} 
&W_1=y_2x_2 y_5 x_8 y_9 y_{12} x_{12} y_{13} x_{16} y_{17} x_{17} x_{20};\\
&X_1=x_1 x_4 x_7 y_8 x_{11} x_{14} y_{15} y_{18} y_{21};\\
&Y_1=x_0 y_1 y_4 x_5 y_6 x_9 x_{10} y_{11} y_{14} x_{15} x_{18} y_{19};\\
&Z_1=y_0 x_3 y_3 x_6 y_7 y_{10} x_{13} y_{16} x_{19};\\
    \end{aligned}
  $ 
\par}
{\centering
  $ \displaystyle
    \begin{aligned} 
&Q_1=x_{20} x_{21} y_{22} x_{23} x_{26} x_{27} y_{28} x_{29} x_{32};\\
&R_1=y_{21} y_{24} y_{27} y_{30} y_{33};\\
&S_1=y_{19} y_{20} y_{23} x_{24} x_{25} x_{28}y_{31};\\
&T_1=x_{19} x_{22} y_{25} y_{26} y_{29} x_{30} x_{31}.
    \end{aligned}
  $ 
\par}
\end{multicols}

Let $L_2=(W_2, X_2, Y_2,  Z_2, Q_2, R_2, S_2, T_2)$ where

\begin{multicols}{2}
{\centering
  $ \displaystyle
    \begin{aligned} 
&W_2=y_2x_5 y_5 x_6 x_9 y_{10} y_{13} x_{13} x_{14} y_{17} y_{20};\\
&X_2=y_1 x_1 y_4 x_7 x_{10} x_{11} y_{14} x_{17} x_{18} y_{21};\\
&Y_2=x_0 y_0 y_3 y_6 y_9 x_{12} x_{15} y_{18} x_{21};\\
&Z_2=x_2 x_3 x_4 y_7 x_8 y_8 y_{11} y_{12} y_{15} y_{16} x_{16} x_{19};
    \end{aligned}
  $ 
\par}
{\centering
  $ \displaystyle
    \begin{aligned} 
&Q_2=y_{20} x_{20} x_{23} y_{23} y_{26} x_{26} x_{29} y_{29} y_{32};\\
&R_2= y_{21} x_{24} y_{27} x_{30} y_{33};\\
&S_2=x_{21} y_{24} x_{27} y_{30} x_{33};\\
&T_2=x_{19} y_{19} x_{22} y_{22} x_{25} y_{25} x_{28} y_{28} x_{31}.
    \end{aligned}
  $ 
\par}
\end{multicols}

For $k=0$, we replace each dipath in $\{Q_i, R_i, S_i, T_i\ |\ i=0,1,2\}$ with a dipath of length $0$ with the same source. It can be verified that the 8-tuples $L_0$, $L_1$, and $L_2$ satisfy conditions C1-C6 of Definition \ref{defn:basic2}, and thus are type-2 basic sets of dipaths.

Let $L_3=(X_3, Y_3, R_3, S_3)$ where

\medskip
{\centering
  $ \displaystyle
    \begin{aligned} 
&X_3=y_0 x_1 y_1 x_2 y_2 y_3 x_4 x_5 x_6 y_6 y_7 x_7 x_8 x_{11} y_{11} x_{12} y_{12} x_{13}x_{16} y_{19};\\
&Y_3=x_0 x_3 y_4 y_5 y_8 x_9 y_9 y_{10} x_{10} y_{13} x_{14} y_{14} y_{15} x_{15}y_{16} x_{17} y_{17} y_{18} x_{18} x_{19} ;\\
&R_3= y_{19} x_{20} y_{20} y_{21} x_{21} x_{22} x_{25} y_{28} y_{29} x_{29} x_{30} y_{30} y_{31};\\
&S_3=x_{19}y_{22} y_{23} x_{23} x_{24} y_{24} y_{25} x_{26} y_{26} y_{27} x_{27} x_{28} x_{31}.
    \end{aligned}
  $ 
\par}
\medskip

Let $L_4=(X_4, Y_4, R_4, S_4)$ where
 
\medskip
{\centering
  $ \displaystyle
    \begin{aligned} 
&X_4=y_1 y_2 x_3y_6 x_6x_7 y_7 y_8 y_9 x_{10} y_{10} y_{11} x_{11} x_{12}x_{13} y_{13} y_{14} x_{14} x_{17} y_{20};\\
&Y_4=y_0 x_0 x_1 x_2 y_3 y_4 x_4y_5 x_5 x_8 x_9 y_{12} x_{15} y_{15} x_{16}y_{16}y_{17} x_{18} y_{18}y_{19};\\
 &R_4=y_{20} x_{21} y_{21} y_{22}x_{22} x_{23} y_{26} x_{27} y_{27} y_{28}x_{28} x_{29} y_{32};\\
&S_4= y_{19} x_{19} x_{20} y_{23} y_{24}x_{24} y_{25} x_{25}x_{26} y_{29} y_{30}x_{30}y_{31}.
    \end{aligned}
  $ 
\par}
\medskip

Again, for $k=0$, we replace each dipath in $\{R_i, S_i \ |\ i=3, 4\}$ with a dipath of length $0$ with the same source. Similarly to the three previous cases, it can be verified that quadruples $L_3$ and $L_4$ satisfy conditions C1-C6 of Definition \ref{defn:basic1}, and thus are type-1 basic sets of dipaths. 

Lastly, it can be verified that $\{L_0, L_1, L_2, L_3, L_4\}$ satisfies the hypothesis of Lemma \ref{lem:list3}. \end{proof}

\section{Proof of the main result}

In this section, we use Lemmas \ref{lem:list2} and \ref{lem:list1}, and Proposition \ref{thm:cand} to construct a $\vec{C}_m$-factorization of $K^*_{2m}$. Theorems \ref{berm} and \ref{Ng2},  and Lemmas \ref{lem:disj} and \ref{lem:deco} below are tools that we will use to obtain the desired construction. 

\begin{theorem} \cite{BerFavMah}
\label{berm}
A connected 4-regular circulant on $m$ vertices admits a $C_m$-decomposition. 
\end{theorem}

\begin{theorem}\cite{Ng} \label{Ng2}
Let $m$ be odd and $n\geqslant 3$. Each of $\vec{C}_m \wr K^*_n$ and $\vec{C}_m \wr \vec{C}_n$ admits a $\vec{C}_{mn}$-decomposition.
\end{theorem}

\begin{lemma}\cite{BurSaj}
\label{lem:disj}
If a digraph $H$ admits a $\vec{C}_m$-factorization, then $kH$ admits a $\vec{C}_m$-factorization. 
\end{lemma}

\begin{lemma} \cite{BurSaj}
\label{lem:deco}
Let $\{H_1, H_2, \ldots, H_k\}$ be a decomposition of a digraph $G$ into spanning sub-digraphs. If each $H_i$ admits a $\vec{C}_m$-factorization, then $G$ admits a $\vec{C}_m$-factorization.
\end{lemma}

In Lemma \ref{lem:2}, we showed that $L_{2m}$ does not admit a $\vec{C}_m$-factorization when $3|m$. We circumvent this case using Lemma \ref{thm:reduction} below. This lemma reduces Problem \ref{prob:ini} to the case $m$ is odd and $3 \nmid m$. 

\begin{lemma} 
\label{thm:reduction}
Let $m$ be odd. If $K^*_{2m}$ admits a $\vec{C}_m$-factorization, then $K^*_{2(3m)}$ admits a $\vec{C}_{3m}$-factorization. 
\end{lemma}

\begin{proof}
Assume that $K^*_{2m}$ admits a $\vec{C}_m$-factorization. Let $F_1, F_2, \ldots, F_{2m-1}$ be the corresponding $\vec{C}_m$-factors, so for each  $k \in \{1, 2, \ldots, 2m\minus1 \}$, we have $F_k \cong 2\vec{C}_m$. Then, we can  obtain the following decomposition of $K^*_{2(3m)}$ into spanning subdigraphs:

\begin{center}
$K^*_{2(3m)}=K^*_{2m}\wr K^*_3 =(F_1\oplus F_2\oplus\ldots \oplus F_{2m-1})\wr K^*_3=F_1 \wr K^*_3\oplus F_2 \wr \overline{K}_3 \oplus \ldots \oplus F_{2m-1} \wr \overline{K}_3.$
\end{center}

Observe that $F_1\wr K^*_3\cong 2(\vec{C}_m\wr K^*_3)$ and $F_i\wr \overline{K_3} \cong 2(\vec{C}_m \wr \overline{K}_3)$ for $i>1$.  Theorem \ref{Ng2}, in conjunction with Lemma \ref{lem:disj}, implies that each of $2(\vec{C}_{m} \wr K^*_3)$ and  $2(\vec{C}_m \wr \overline{K}_3)$ admits a $\vec{C}_{3m}$-factorization, respectively. In conclusion, Lemma \ref{lem:deco} implies that $K^*_{2(3m)}$ admits a $\vec{C}_{3m}$-factorization. \end{proof}

Observe that Lemma \ref{thm:reduction} is vacuous when $m=3$ because $K^*_6$ does not admit a $\vec{C}_3$-factorization \cite{BerGerSot}. 

Now, we proceed with the proof of our main theorem, Theorem \ref{thm:main}, restated below for convenience.

\noindent {\bf Theorem 3} Let $m$ be an odd integer such that $m \geqslant 5$. The digraph $K^*_{2m}$ admits a $\vec{C}_m$-factorization.
\\

\begin{proof}
First, assume that $m \equiv 1, 5\ (\textrm{mod}\ 6)$ and $m \geqslant 13$. We begin by strategically decomposing the graph $K_m$. If $m \equiv 1\ (\textrm{mod}\ 4)$, then
 \begin{center}
$K_m=X(m, \{1,3\})\oplus X(m, \{2,4\})\oplus X(m, \{5,6\})\oplus \ldots \oplus X(m, \{\frac{m-3}{2}, \frac{m-1}{2}\})$. 
\end{center} 
 \noindent If $m \equiv 3\ (\textrm{mod}\ 4)$, then 
 \begin{center}
 $K_m=X(m, \{1,3\})\oplus X(m, \{2\})\oplus X(m, \{4,5\})\oplus \ldots \oplus X(m, \{\frac{m-3}{2}, \frac{m-1}{2}\})$. 
\end{center}  
 \noindent Since each 4-regular circulant in this decomposition of $K_m$ is connected, Theorem \ref{berm} implies that $K_m$ can be decomposed into one copy of $X(m, \{1,3\})$ and $\frac{m-5}{2}$ Hamiltonian cycles. Each Hamiltonian cycle of $K_m$ is isomorphic to $X(m, \{1\})$. It follows that $K^*_m$ admits a decomposition into one copy of $\vec{X}(m, \{\pm1,\pm3\})$ and $\frac{m-5}{2}$ copies of $\vec{X}(m, \{\pm 1\})$. Consequently, we see that $K_{2m}^*=K^*_m \wr K^*_2$ admits a decomposition into $\frac{m-5}{2}$ copies of $\vec{X}(m, \{\pm 1\})\wr \overline{K}_2 \cong H_{2m}$, and a copy of $\vec{X}(m, \{\pm1, \pm 3\}) \wr K^*_2 \cong L_{2m}\oplus G_{2m}$. Since $H_{2m}, L_{2m}$, and $G_{2m}$ each admit a $\vec{C}_m$-factorization by Lemmas \ref{lem:1} and \ref{lem:2}, and Proposition \ref{thm:cand}, respectively, it follows that $K^*_{2m}$ admits a $\vec{C}_m$-factorization. 

Finally, assume that $m \equiv 3\ (\textrm{mod}\ 6)$ or $5 \leqslant m \leqslant 11$. If $m=3^rt$ for some $t\equiv 1, 5\ (\textrm{mod}\ 6)$, $t \geqslant 13$, and $r \geqslant 1$, then a $\vec{C}_t$-factorization of $K^*_{2t}$ exists by the above, and the existence of a $\vec{C}_m$-factorization of $K^*_{2m}$ is established by a repeated application of Lemma \ref{thm:reduction}.

Otherwise, if $m=3^rt$ for $t\in \{5,7,9,11\}$, and $r=0$ then a $\vec{C}_t$-factor of $K^*_{2t}$ exists by Theorem \ref{BurFraSaj} \cite{BurFranSaj}, and if $r\geqslant 1$, then a repeated application of Lemma \ref{thm:reduction} can be used to show existence of a $\vec{C}_m$-factorization of $K_{2m}^*$.
\end{proof}

\medskip

\noindent{\bf \large Acknowledgements}

\noindent The author would like to thank her PhD supervisor Mateja \v{S}ajna for her guidance and support over the course of this project. The author would also like to thank the NSERC CGS-D scholarship program for its financial support.

\appendix

\section{Illustrating the proof of Proposition \ref{thm:cand}}

We illustrate all dipaths built in the proof of Proposition \ref{thm:cand}. Recall that the edges are assumed to be oriented from left to right, and vertical arcs are oriented to yield a dipath. 

\noindent {\bf Case 1:} $p=11$. 

\begin{figure}[H]
\begin{subfigure}[c]{1 \textwidth}


\caption{The dipaths $W_0$ (pink), $X_0$ (red), $Y_0$ (dark blue), and $Z_0$ (light blue).}
\end{subfigure}

\begin{subfigure}[c]{1 \textwidth}

%
\caption{The dipaths $Q_0$ (pink), $R_0$ (red), $S_0$ (dark blue), and $T_0$ (light blue).}
\label{fig:m11.1.2}
\end{subfigure}
\caption{The type-2 basic set of dipaths $L_0$ for $p=11$.}
\label{fig:m11.1.1}
\end{figure}

\begin{figure}[H]
\begin{subfigure}[c]{1 \textwidth}
%
\caption{The dipaths  $W_1$ (pink), $X_1$ (red), $Y_1$ (dark blue), and $Z_1$ (light blue).}
\label{fig:m11.2.1}
\end{subfigure}

\begin{subfigure}[c]{1 \textwidth}
%
\caption{The dipaths  $Q_1$ (pink), $R_1$ (red), $S_1$ (dark blue), and $T_1$ (light blue).}
\label{fig:m11.2.2}
\end{subfigure}
\caption{The type-2 basic set of dipaths $L_1$ for $p=11$.}
\end{figure}

\begin{figure}[H]
\begin{subfigure}[c]{1 \textwidth}
%
\caption{The dipaths  $W_2$ (pink), $X_2$ (red), $Y_2$ (dark blue), and $Z_2$ (light blue).}
\label{fig:m11.3.1}
\end{subfigure}

\begin{subfigure}[c]{1 \textwidth}
%
\caption{The dipaths $Q_2$ (pink), $R_2$ (red), $S_2$ (dark blue), and $T_2$ (light blue).}
\label{fig:m11.3.2}
\end{subfigure}
\caption{The type-2 basic set of dipaths $L_2$ for $p=11$.}
\end{figure}

\begin{figure}[H]
\begin{subfigure}[c]{1 \textwidth}
%
\caption{The dipaths $X_3$ (red) and $Y_3$ (dark blue).}
\label{fig:m11.4.1}
\end{subfigure}

\begin{subfigure}[c]{1 \textwidth}
%
\caption{The dipaths $R_3$ (red) and $S_3$ (dark blue).}
\label{fig:m11.4.2}
\end{subfigure}
\caption{The type-1 basic set of dipaths $L_3$ for $p=11$.}
\end{figure}

\begin{figure}[H]
\begin{subfigure}[c]{1 \textwidth}
%
\caption{The dipaths  $X_4$ (red) and $Y_4$ (dark blue).}
\end{subfigure}

\begin{subfigure}[c]{1 \textwidth} 
%
\caption{The dipaths $R_4$ (red) and $S_4$ (dark blue). }
\end{subfigure}
\caption{The type-1 basic set of dipaths $L_4$ for $p=11$.}
\label{fig:m11.5.2}
\end{figure}

\noindent {\bf Case 2:} $p=13$.\\

\begin{figure}[H]
\begin{subfigure}[c]{1.0\textwidth}
%
\caption{The dipaths  $W_0$ (pink), $X_0$ (red), $Y_0$ (dark blue), and $Z_0$ (light blue).}

\end{subfigure}

\begin{subfigure}[c]{1.0\textwidth}
%
\caption{The dipaths  $Q_0$ (pink), $R_0$ (red), $S_0$ (dark blue), and $T_0$ (light blue).}
\label{fig:m13.1.2}
\end{subfigure}
\caption{The type-2 basic set of dipaths $L_0$ for $p=13$.}
\label{fig:m13.1.1}
\end{figure}

\begin{figure}[H]
\begin{subfigure}[c]{1\textwidth}
%
\caption{The dipaths  $W_1$ (pink), $X_1$ (red), $Y_1$ (dark blue), and $Z_1$ (light blue).}
\label{fig:m13.2.1}
\end{subfigure}

\begin{subfigure}[c]{1\textwidth}

%
\caption{The dipaths  $Q_1$ (pink), $R_1$ (red), $S_1$ (dark blue), and $T_1$ (light blue).}
\label{fig:m13.2.2}
\end{subfigure}
\caption{The type-2 basic set of dipaths $L_1$ for $p=13$.}
\end{figure}

\begin{figure}[H]
\begin{subfigure}[c]{1\textwidth}
%
\caption{The dipaths  $W_2$ (pink), $X_2$ (red), $Y_2$ (dark blue), and $X_2$ (light blue).}
\label{fig:m13.3.1}
\end{subfigure}

\begin{subfigure}[c]{1 \textwidth}
%
\caption{The dipaths  $Q_2$ (pink), $R_2$ (red), $S_2$ (dark blue), and $T_2$ (light blue).}
\label{fig:m13.3.2}
\end{subfigure}
\caption{The type-2 basic set of dipaths $L_2$ for $p=13$.}
\end{figure}

\begin{figure}[H]
\begin{subfigure}[c]{1\textwidth}
%
\caption{The dipaths   $X_3$ (red) and $Y_3$ (dark blue).}
\label{fig:m13.4.1}
\end{subfigure}

\begin{subfigure}[c]{1 \textwidth}
%
\caption{The dipaths $R_3$ (red) and $S_3$ (dark blue).}
\label{fig:m13.4.2}
\end{subfigure}
\caption{The type-1 basic set of dipaths $L_3$ for $p=13$.}
\end{figure}

\begin{figure}[H]
\begin{subfigure}[c]{1\textwidth}
%
\caption{The dipaths $X_4$ (red) and $Y_4$ (dark blue).}
\end{subfigure}

\begin{subfigure}[c]{1\textwidth}

%
\caption{The dipaths $R_4$ (red) and $S_4$ (dark blue).}
\end{subfigure}
\caption{The type-1 basic set of dipaths $L_4$ for $p=13$.}
\label{fig:m13.5.2}
\end{figure}

\pagebreak
\noindent {\bf Case 3:} $p=17$. 

\begin{figure}[H]
\begin{subfigure}[c]{1\textwidth}
%
\caption{The dipaths $W_0$ (pink), $X_0$ (red), $Y_0$ (dark blue), and $Z_0$ (light blue).}
\end{subfigure}

\begin{subfigure}[c]{1\textwidth}

%
\caption{The dipaths $Q_0$ (pink), $R_0$ (red), $S_0$ (dark blue), and $T_0$ (light blue).}
\label{fig:m17.1.2}
\end{subfigure}
\caption{The type-2 basic set of dipaths $L_0$ for $p=17$.}
\label{fig:m17.1.1}
\end{figure}

\begin{figure}[H]
\begin{subfigure}[c]{1\textwidth}

%
\caption{The dipaths $W_1$ (pink), $X_1$ (red), $Y_1$ (dark blue), and $Z_1$ (light blue).}
\label{fig:m17.2.1}
\end{subfigure}

\begin{subfigure}[c]{1\textwidth}

%
\caption{The dipaths  $Q_1$ (pink), $R_1$ (red), $S_1$ (dark blue), and $T_1$ (light blue).}
\end{subfigure}
\caption{The type-2 basic set of dipaths $L_1$ for $p=17$.}
\label{fig:m17.2.2}
\end{figure}

\begin{figure}[H]
\begin{subfigure}[c]{1\textwidth}

%
\caption{The dipaths $W_2$ (pink), $X_2$ (red), $Y_2$ (dark blue), and $Z_2$ (light blue).}
\label{fig:m17.3.1}
\end{subfigure}

\begin{subfigure}[c]{1\textwidth}

%
\caption{The dipaths  $Q_2$ (pink), $R_2$ (red), $S_2$ (dark blue), and $T_2$ (light blue).}
\end{subfigure}
\caption{The type-2 basic set of dipaths $L_2$ for $p=17$.}
\label{fig:m17.3.2}
\end{figure}

\begin{figure}[H]
\begin{subfigure}[c]{1\textwidth}

%
\caption{The dipaths  $X_3$ (red) and $Y_3$ (dark blue).}
\label{fig:m17.4.1}
\end{subfigure}

\begin{subfigure}[c]{1\textwidth}
%
\caption{The dipaths $R_3$ (red) and $S_3$ (dark blue).}
\label{fig:m17.4.2}
\end{subfigure}
\caption{The type-1 basic set of dipaths $L_3$ for $p=17$.}
\end{figure}

\begin{figure}[H]
\begin{subfigure}[c]{1\textwidth}
%
\caption{The dipaths $X_4$ (red) and $Y_4$ (dark blue).}
\end{subfigure}

\begin{subfigure}[c]{1\textwidth}
%
\caption{The dipaths  $R_4$ (red) and $S_4$ (dark blue).}
\end{subfigure}
\caption{The type-1 basic set of dipaths $L_4$ for $p=17$.}
\label{fig:m17.5.2}
\end{figure}

\noindent {\bf Case 4:} $p=19$.

\begin{figure}[H]
\begin{subfigure}[c]{1.00 \textwidth}
%
\caption{The dipaths  $W_0$ (pink), $X_0$ (red), $Y_0$ (dark blue), and $Z_0$ (light blue).}
\end{subfigure}

\begin{subfigure}[c]{1.0\textwidth}
%
\caption{The dipaths  $Q_0$ (pink), $R_0$ (red), $S_0$ (dark blue), and $T_0$ (light blue).}
\label{fig:m19.1.2}
\end{subfigure}
\caption{The type-2 basic set of dipaths $L_0$ for $p=19$.}
\label{fig:m19.1.1}
\end{figure}

\begin{figure}[H]
\begin{subfigure}[c]{1.00\textwidth}
%
\caption{The dipaths $W_1$ (pink), $X_1$ (red), $Y_1$ (dark blue), and $Z_1$ (light blue).}
\label{fig:m19.2.1}
\end{subfigure}

\begin{subfigure}[c]{1 \textwidth}

%
\caption{The dipaths $Q_1$ (pink), $R_1$ (red), $S_1$ (dark blue), and $T_1$ (light blue).}
\label{fig:m19.2.2}
\end{subfigure}
\caption{The type-2 basic set of dipaths $L_1$ for $p=19$.}
\end{figure}

\begin{figure}[H]
\begin{subfigure}[c]{1.00\textwidth}

%
\caption{The dipaths  $Q_2$ (pink), $R_2$ (red), $S_2$ (dark blue), and $T_2$ (light blue).}
\label{fig:m19.3.2}
\end{subfigure}
\caption{The type-2 basic set of dipaths $L_2$ for $p=19$.}
\end{figure}

\begin{figure}[H]
\begin{subfigure}[c]{1\textwidth}
%
\caption{The dipaths  $X_3$ (red) and $Y_3$ (dark blue).}
\label{fig:m19.4.1}
\end{subfigure}

\begin{subfigure}[c]{1\textwidth}
%
\caption{The dipaths $R_3$ (red) and $S_3$ (dark blue).}
\label{fig:m19.4.2}
\end{subfigure}
\caption{The type-1 basic set of dipaths $L_3$ for $p=19$.}
\end{figure}

\begin{figure}[H]
\begin{subfigure}[c]{1\textwidth}

%
\caption{The dipaths $X_4$ (red) and $Y_4$ (dark blue).}
\end{subfigure}

\begin{subfigure}[c]{1 \textwidth}

%
\caption{The dipaths $R_4$ (red) and $S_4$ (dark blue)}

\end{subfigure}
\caption{The type-1 basic set of dipaths $L_4$ for $p=19$.}
\label{fig:m19.5.2}
\end{figure}


\begin{thebibliography}{99}
\nocite{*}

\bibitem{Abel} R. J. R. Abel, F. E. Bennett, G. Ge,  Resolvable perfect Mendelsohn designs with block size 5, {\em Discrete Math.}, 247:1--12, 2002.

\bibitem{AdaBry} P. Adams, D. Bryant, Resolvable directed cycle systems of all indices for cycle length 3 and 4, unpublished.

\bibitem{AlsHag} B. Alspach, R. H\"aggkvist,  Some observations on the Oberwolfach problem, {\em J. Graph Theory}, 9:177--187, 1985.

\bibitem{AlsGavSaj} B. Alspach, H. Gavlas, M. \v{S}ajna, H. Verrall,  Cycle decompositions IV: Complete directed graphs and fixed length directed cycles, {\em J. Combin. Theory Ser. A}, 103:165--208, 2003.

\bibitem{AlsSch} B. Alspach, P. J. Schellenberg, D. R. Stinson, D. Wagner, The Oberwolfach problem and factors of uniform odd length cycles, {\em J. Combin. Theory Ser. A}, 52:20--43, 1989.

\bibitem{BenZha} F. E. Bennett, X. Zhang,  Resolvable Mendelsohn designs with block size $4$, {\em Aequationes Math.}, 40:248--260, 1990.

\bibitem{BerFavMah} J.-C. Bermond, O. Favaron, and M. Mah\'{e}o,  Hamiltonian decomposition of Cayley graphs of degree 4, {\em J. Combin. Theorey Ser. B},  46:142--153, 1989.

\bibitem{BerGerSot} J.-C. Bermond, A. Germa, D. Sotteau, Resolvable decomposition of $K_n^\ast$,
{\em J. Combin. Theory Ser. A}, 26:179--185, 1979.

\bibitem{BurFranSaj} A. Burgess, N. Franceti\'{c}, M. \v{S}ajna, On the directed {O}berwolfach {P}roblem with equal cycle lengths: the odd case, {\em Australas. J. Comb.}, 71:272--292, 2018.

\bibitem{BurSaj} A. Burgess and M. \v{S}ajna,  On the directed {O}berwolfach {P}roblem with equal cycle lengths, {\em Electon. J. Combin.}, 21:1--15, 2014.

\bibitem{ColDin} C. J. Colbourn, J. H. Dinitz (editors), {\em Handbook of combinatorial designs,} Chapman and Hall/CRC, Boca Raton, FL, 2007.

\bibitem{HofSch} D. G. Hoffman, P. J. Schellenberg, The existence of $C_k$-factorizations of $K_{2n}-F$, {\em Discrete Math.}, 97:243--250, 1991.

\bibitem{HuaKot} C. Huang, A. Kotzig, A. Rosa, On a variation of the Oberwolfach problem, {\em Discrete Math.}, 27:261--277, 1979.

\bibitem{Ng} L. Ng, Hamiltonian Decomposition of Lexicographic Products of Digraphs, {\em J. Combin. Theory Ser. B},  73:119--129, 1998.

\bibitem{Til} T. W. Tillson, A hamiltonian decomposition of $K_{2m}^\ast$, $2m \ge 8$, {\em J. Combin. Theory Ser. B}, 29:69--74, 1980.

\end{thebibliography}
\end{document}